\documentclass{amsart}
\pdfoutput=1
\usepackage{amssymb, latexsym, amsmath, verbatim, amsthm, amscd}
\usepackage{pinlabel}
\usepackage[all]{xy} 
\usepackage{graphicx}

\usepackage{hyperref}

\def\mytitle{Line Patterns in Free Groups}
\def\mykeywords{free group, quasi-isometry, rigidity, line pattern, Whitehead graph, Whitehead's Algorithm}
\title[\mytitle]{\mytitle}

\author{Christopher H. Cashen} 
\address{Christopher H. Cashen\newline \indent
Department of Mathematics\newline \indent
University of Utah\newline \indent 
Salt Lake City, UT 84112\\USA} 
\email{\href{mailto:cashen@math.utah.edu}{cashen@math.utah.edu}}
\urladdr{\href{http://www.math.utah.edu/~cashen}{http://www.math.utah.edu/~cashen}}

\author{Nata\v{s}a Macura} 
\address{Nata\v{s}a Macura\newline \indent
Department of Mathematics\newline \indent
Trinity University\newline \indent 
San Antonio, TX 78212\\USA} 
\email{\href{mailto:nmacura@trinity.edu}{nmacura@trinity.edu}}

\hypersetup{
    pdftitle={\mytitle},    
    pdfauthor={Christopher H. Cashen and Natasa Macura},     
    pdfkeywords={\mykeywords}, 
    colorlinks=true,       
    linkcolor=black,          
    citecolor=black,        
    filecolor=black,      
    urlcolor=black           
}

\theoremstyle{plain}
\newtheorem{theorem}{Theorem}[section]
\newtheorem{lemma}{Lemma}[section]
\newtheorem{proposition}{Proposition}[section]
\newtheorem{corollary}{Corollary}[section]

\newtheorem*{maintheorem}{ Main Theorem}
\newtheorem*{reftheorem}{Theorem}
\newtheorem{question}{Question}

\theoremstyle{remark}

\newtheorem*{note}{Note}
\newtheorem*{remark}{Remark}

\theoremstyle{definition}
\newtheorem{definition}{Definition}[section]

\def\makeautorefname#1#2{\expandafter\def\csname#1autorefname\endcsname{#2}}
\let\fullref\autoref

\makeautorefname{theorem}{Theorem} 
\makeautorefname{lemma}{Lemma} 
\makeautorefname{proposition}{Proposition} 
\makeautorefname{corollary}{Corollary} 
\makeautorefname{definition}{Definition}
\makeautorefname{example}{Example}
\makeautorefname{section}{Section}
\makeautorefname{subsection}{Section}
\makeautorefname{subsubsection}{Section}
\makeautorefname{question}{Question}
\makeautorefname{conjecture}{Conjecture}

\makeatletter 
\let\c@lemma=\c@theorem 
\makeatother
\makeatletter 
\let\c@proposition=\c@theorem 
\makeatother
\makeatletter 
\let\c@corollary=\c@theorem 
\makeatother
\makeatletter 
\let\c@definition=\c@theorem 
\makeatother
\makeatletter 
\let\c@example=\c@theorem 
\makeatother

\DeclareMathOperator{\shadow}{shadow}
\DeclareMathOperator{\Hy}{\mathbb{H}} 
\DeclareMathOperator{\QIgp}{\mathcal{QI}} 
\DeclareMathOperator{\Isomgp}{Isom} 
\DeclareMathOperator{\Sym}{Sym} 
\DeclareMathOperator{\Aut}{Aut} 
\DeclareMathOperator{\Wh}{Wh}
\def\bdry{\partial}
\def\haus{d_H}
\def\ray{\mathcal{R}} 
\def\X{\mathcal{X}} 
\def\Y{\mathcal{Y}}
\def\core{\mathcal{C}}
\def\pcore{p\core} 
\def\ppcore{pp\core} 
\def\hull{\mathcal{H}} 
\def\tree{\mathcal{T}} 
\def\lines{\mathcal{L}} 
\def\decomp{\textbf{D}} 
\def\e{\epsilon}
\newcommand{\closure}[1]{\overline{#1}} 

\newlength{\figstandardheight}
\def\from{\colon\thinspace}

\begin{document}
\begin{abstract}
We study line patterns in a free group by considering the topology of
the decomposition space, a quotient of the boundary at infinity of the
free group related to the line pattern.
We show that the group of quasi-isometries preserving a line pattern
in a free group acts by isometries on a related space if and only if there
are no cut pairs in the decomposition space.
\end{abstract}

\maketitle

\section{Introduction}
Given a finitely generated free group $F$ of rank greater than one and a word $w\in F$, the $w$-line at $g\in F$ is the set of
elements $\{gw^m\}_{m\in \mathbb{Z}}$.
Up to translation and coarse equivalence, 
we may assume that $w$ is cyclically reduced and not a power of
another element.
A Cayley graph with respect to a free basis of $F$ is a geometric
model for $F$ that is a tree, and in this case there is a unique geodesic in
the tree that contains the vertices $\{gw^m\}_{m\in \mathbb{Z}}$.

The $w$-line at $g$ is the same as the $w$ line at $h$ if
and only if $\bar{h}g$ is a power of $w$; the $w$-lines are the cosets
of $\left<w\right>$ in $F$.

The \emph{line pattern generated by $w$} is the collection of
distinct $w$-lines. Similarly, if we take finitely many words $w$, as
above, the line pattern generated by the collection is the union of
the patterns generated by the individual words. We will denote the line pattern $\lines$ when we do not wish to specify generators.

The main question is:
\begin{question}\label{question:qiequivalence}
  Let $F$ and $F'$ be finite rank free groups, possibly of different
  ranks.
Consider collections of words
  $\{w_1,\,\dots , w_m\}\subset F$ and $\{w'_1,\,\dots , w'_n\}\subset
  F'$. Let $\lines$ be the line pattern in $F$ generated by $\{w_1,\,\dots , w_m\}$, and let $\mathcal{L'}$ be defined similarly for $F'$. 

Is there a quasi-isometry $\phi\from F \to F'$ that preserves
  the patterns, in the sense that there is some constant $C$ so that for every line 
$l\in\lines$ there is an $l'\in\mathcal{L'}$ such that the Hausdorff distance between $\phi(l)$ and $l'$ is at most $C$, and vice versa?
\end{question}

A closely related question is:
\begin{question}\label{question:qigp}
  Let $F$ be a free group and $\lines$ a line pattern in $F$. 
What is the group $\QIgp(F,\lines)$ of quasi-isometries of $F$ that preserve the line pattern $\lines$?
\end{question}

In a pair of 1936  papers \cite{Whi36a, Whi36}, J. H. C. Whitehead
gave an algorithm to answer the following question:

Given two finite (ordered) lists of words $(w_1,\dots, w_k)$ and $(w_1',\dots
w_k')$ in a finite rank free group $F$, is there an automorphism
$\phi$ of $F$ such that for all $i$, $\phi(w_i)=w_i'$?

Questions~\ref{question:qiequivalence} and \ref{question:qigp} may
be viewed as geometric versions of Whitehead's question.

To motivate the statement of our results, it is instructive to consider line patterns
in a different setting. 

Line patterns in $\Hy^n$ for $n\geq 3$ have been studied by Schwartz
\cite{Sch97}. 
His terminology is ``symmetric pattern of geodesics''.
Let $M$ be a compact hyperbolic orbifold of dimension $n\geq 3$.
Pick any collection of closed geodesics in $M$. 
The lifts of these geodesics to the universal cover $\Hy^n$ are a
line pattern; call it $\lines$.

\begin{reftheorem}{\rm {\cite[Theorem~1.1]{Sch97}}}
  \[\QIgp(\Hy^n,\lines)\subset \Isomgp(\Hy^n)\]
\end{reftheorem}

This is an example of what we will call \emph{pattern rigidity}. 
The hyperbolic orbifold case is special in that there is a canonical
geometric model, $\Hy^n$, for $\pi_1M$.
Forgetting this for a moment, let $Y$ be any geometric model for
$\pi_1M$. 
For example, $Y$ could be a Cayley graph of $\pi_1M$. 
We still get a line pattern $\lines$ in $Y$, but it is not
necessarily true that $\QIgp(Y,\lines)\subset\Isomgp(Y)$.
However, there is a quasi-isometry $\phi\from Y \to \Hy^n$. 
Each line in $\lines$ gets sent to a line in $\Hy^n$, so we get a
line pattern $\phi(\lines)$ in $\Hy^n$. 
We have:
\[\phi\QIgp(Y,\lines)\phi^{-1}=\QIgp(\Hy^n,\phi(\lines))\subset\Isomgp(\Hy^n)\]

In the free group situation we do not have a canonical space to take
the place of $\Hy^n$ that works for every line pattern.
For a given line pattern we will try to construct a space $X$
and a
quasi-isometry $\phi\from F\to X$ such that pattern preserving
quasi-isometries are conjugate into the isometry group of $X$:
\[\phi\QIgp(F,\lines)\phi^{-1}=\QIgp(X,\phi(\lines))\subset\Isomgp(X)\]

A priori this would only give a quasi-action of
$\QIgp(F,\lines)$ on $X$ by maps bounded distance from isometries.
We actually prove something stronger. 
We will say a line pattern
$\lines$ in $F$ is \emph{rigid} if there is a space $X$,
a
quasi-isometry $\phi\from F\to X$, and an isometric action of
$\QIgp(F,\lines)$ on $X$ that agrees with
conjugation by $\phi$, up to bounded distance.

It is easy to see that not all patterns are rigid. 
A necessary
condition is that the multiplicative  quasi-isometry constants of
$\QIgp(F,\lines)$ are bounded. 
Suppose $\lines$ is contained
in a proper free factor $F'$ of $F$, so that
$F=F'*F''$. 
Then $\QIgp(F'')\subset\QIgp(F,\lines)$ contains
a sequence of quasi-isometries with unbounded constants, so the
pattern is not rigid.

\hypertarget{introexample}{
Another example where the lack of rigidity is apparent for algebraic
reasons is the pattern generated by the word $ab\bar{a}\bar{b}$ in
$F_2=\left<a,b\right>$.
The automorphism group of $F_2$ preserves this line pattern, so again
we have a sequence of pattern preserving quasi-isometries with unbounded constants.}

However, algebraic considerations do not fully determine which
patterns are rigid. 
Consider the pattern in $F_2$ generated by $ab$
and $a\bar{b}$. 
There is only a finite group of outer automorphisms of
$F_2$ that preserve this pattern, so all pattern preserving
automorphisms are isometries, up to bounded distance.
We might guess the pattern is rigid, but in fact it
 is quasi-isometrically equivalent to the $ab\bar{a}\bar{b}$
pattern, see \fullref{theorem:allcircles}. 

Our main result shows that sufficiently
complicated patterns are rigid. 
To make this precise, we use a
topological space that is a quotient of the boundary at infinity of a
tree for $F$. 
This space is called the \emph{decomposition space} associated to
the line pattern. 

\begin{maintheorem}\hypertarget{main}{
Let $\lines$ be a line pattern in a finitely generated,
non-abelian free group, $F$. The following are equivalent:
\begin{enumerate}
\item The line pattern is rigid.
\item The decomposition space has no cut pairs.
\end{enumerate}}
\end{maintheorem}

\begin{remark}
  We use the phrase ``has no cut pairs'' inclusively to mean that the space is
  connected, has no cut points and no cut pairs.
\end{remark}

In \fullref{sec:flexibility} we show that when the decomposition space
has cut pairs there is a sequence of
pattern preserving quasi-isometries with unbounded quasi-isometry
constants, so the pattern is not rigid. 
We also show in this case that $F$ is not finite index in $\QIgp(F,\lines)$.

In the examples above, the pattern that is contained in a proper free factor would have a
disconnected decomposition space. 
For the other two, the decomposition
space is a circle.

Determining if the decomposition
space is connected is essentially Whitehead's Algorithm, which is
discussed in \fullref{sec:WA}. 
The idea is to build a graph, the
Whitehead graph, associated to the line pattern. 
Connectivity of this
graph is related to connectivity of the decomposition space, see
\fullref{theorem:Bestvina}.

In \fullref{sec:Whandtopology} we use generalizations of the Whitehead
graph to identify finite cut sets in the decomposition space. 
In
particular, \fullref{theorem:findcutpairs} allows us to tell if there are
cut pairs in the decomposition space.

The proof of the rigidity part of the theorem in
\fullref{sec:rigidity} is similar in philosophy to the various
geometric proofs of Stallings' Theorem, see Dunwoody \cite{Dun85},
Gromov \cite{Gro87}, Niblo \cite{Nib04} or Kapovich \cite{Kap07a}.
The idea in these proofs is to use minimal surfaces, or a combinatorial
approximation thereof, to cut up a space into pieces.
One then uses properties of the particular choice of surfaces to show
that they are, or can be chosen to be, suitably independent, so that
the complex dual to the cutting surfaces is a tree.

We do something similar with small cut sets in the decomposition space.
A novel feature of our approach is that the argument takes place ``at
infinity''.
The cut sets we use have more
complicated interactions than those in Stallings' Theorem, and in
general the space dual to the collection of cut sets will not be a tree, it will be a cube complex
quasi-isometric to a tree.

Working at infinity has the benefit that the cube complex we
construct is canonical and inherits a
canonical line pattern. 
$\QIgp(F,\lines)$ is conjugate to the
group of isometries of the cube complex that preserve the line
pattern, see \fullref{theorem:rigidity}.

This allows us to answer Questions 1 and 2 in the rigid case:
Two line patterns in free groups are equivalent if and only if there
is a pattern preserving isometry between the associated cube
complexes. 
The free group $F$ acts cocompactly by pattern
preserving isometries on the cube complex, so $\QIgp(F,\lines)$
does as well. 
This allows us to give a
description of $\QIgp(F,\lines)$ as a complex of
groups. However, the vertex stabilizers will not, in general, be
finitely generated groups.

Consideration of line pattern preserving quasi-isometries arises
naturally in Geometric Group Theory. 
Work of Papasoglu \cite{Pap05} shows that group splittings of finitely
presented groups over virtually cyclic subgroups are preserved by quasi-isometries. 
If a finitely presented, one-ended group has a non-trivial JSJ--decomposition
over virtually cyclic subgroups, then each vertex group of the decomposition has a line
pattern coming from the incident edge groups. 
The equivalence classes of these line patterns give quasi-isometry
invariants for the group, and, in the rigid case, impose severe
restrictions on quasi-isometries of the group.

In particular, the authors came upon this problem in the course of
studying mapping tori of free group automorphisms. In the case of
a linearly growing automorphism, the mapping torus has a
JSJ-decomposition with vertex groups $F\times\mathbb{Z}$.
Understanding the line patterns in the free factors of the vertex
groups is a key step in the
quasi-isometry classification of these mapping tori \cite{CasMac09b}.

\section{Preliminaries}

\subsection{Cut Sets and Cubings}\label{sec:cubings}
If $X$ is a topological space, a \emph{cut set} is a subset $S\subset
X$ such that $X\setminus S=\{x\in X\mid x\notin S\}$ is disconnected. 
A single point that is a cut set is a \emph{cut point}; a pair of points that is a cut set is a \emph{cut pair}, etc. 

A cut set $S$ is \emph{minimal} if no proper subset of $S$ is a cut set of $X$.

If $S$ and $S'$ are cut sets of $X$ we say $S'$ \emph{crosses} $S$ if $S'\setminus S$ has points in multiple components of $X\setminus S$. 
This is not a symmetric relation, but it is if we assume that $S$ and
$S'$ are minimal.

A \emph{cubing} is a simply connected, non-positively curved cube
complex. Cubings can be used to encode the combinatorics of a
collection of cut sets.
Our treatment of cubings is based on work of Sageev
\cite{Sag95}.

Let $\{S_i\}_{i\in I}$ be a collection of closed, minimal cut sets of $X$ so that for
each $i$, $X\setminus S_i$ has exactly two connected components,
$A_i^0$ and $A_i^1$. 
We will take the superscripts mod 2, so that the two components of
$X\setminus S_i$ are $A_i^\e$ and $A_i^{1+\e}$ for $\e\in\{0,1\}$.
Let \[\Sigma=\{A^0_i\}_{i\in I}\cup\{A^1_i\}_{i\in I}\]


Define a cube complex as follows.
The vertices are the subsets $V$ of $\Sigma$ such that:
\begin{enumerate}
\item For all $i\in I$ exactly one of $A^0_i$ or $A^1_i$ is in $V$.
\item If $C\in V$ and $C'\in\Sigma$ with $C\subset C'$ then $C'\in V$. 
\end{enumerate}

Two vertices are connected by an edge if they differ by only one set in $\Sigma$. 

One can identify $\Sigma$ with $2^I$. The $i$-th
``coordinate'' is either 0 for
$A^0_i$ or 1 for $A^1_i$. Edges join vertices that differ in exactly one coordinate.

The vertices are the elements of $2^I$ that are ``consistent'' with
the cut set structure in the sense that if for some $i$ and $j$ we
have $A^1_i\subset A^1_j$  then we do not have any vertices that are
``1'' in the $i$-th coordinate and ``0'' in the $j$-th coordinate.
 It is not consistent to
be simultaneously in $A^1_i$ and $A^0_j$.

Informally, having $\e$ in the $i$-th coordinate corresponds to being
in $A_i^\e$. There is a subtlety here, though. An element of $2^I$ might be consistent
without being realized as a component of $X\setminus \{S_i\}$.
It is possible that there are vertices such that $\e_i$ is the value
of the $i$-th coordinate of the vertex, but $\cap_{i\in I} A_i^{\e_i}=\emptyset$.

\begin{remark}
  There is a minor difference from Sageev's contruction. In his
  notation we would be considering $A_i=S_i\cup A_i^0$ and
  $A_i^c=A_i^1$. The nature of the cut sets we are interested in would make it problematic
  to include them in one of the components. There is only one
  place where this requires us to change Sageev's arguments, which we
  will point out shortly. Everywhere else, it is sufficient to replace
  a statement like:
\[A_i\subset A_j\implies A_j^c\subset A_i^c\]
with a statement like:
\[A_i^{\e_i}\subset A_j^{\e_j}\implies A_j^{1+\e_j}\subset
A_i^{1+\e_i}\]
This statement follows easily from the fact that minimal cut sets are
either mutually crossing or mutually non-crossing.
\end{remark}

Edges in the complex correspond to changing one coordinate from 0 to
1, or vice versa. However, to maintain consistency not every
coordinate can be changed:

\begin{lemma}[{\cite[Lemma 3.2]{Sag95}}]
  If $V$ is a vertex and $A_i^\e\in V$ then
  \[W=(V\setminus\{A_i^\e\})\cup\{A_i^{1+\e}\}\] is a vertex if and only
  if $A_i^\e$ is minimal in $V$, in the sense that $A_i^\e$ does not
  contain any other $A_j^\delta\in V$.
\end{lemma}

It turns out in general that there are still too many vertices. The
graph that has been constructed so far is not necessarily
connected. This is where our construction differs from Sageev's. 
For both his construction and ours, the idea is to select a
subcollection of the vertices, show that the subcollection belongs to
a path connected subset of the graph, and then throw away everything
not in that path component. Our construction will come later in
\fullref{sec:rigidity}. However, this is the only place in which
Sageev uses the special properties of his chosen collection $\Sigma$.
The rest of his arguments go through unchanged in our
setting. 

So assume that we have passed to a non-trivial path connected component of the original graph.
Following Sageev again, one glues in one square (2 dimensional cube)
whenever one sees the boundary of a square in the graph. One proceeds
by induction to glue in an $n$--cube whenever one sees the boundary of
an $n$--cube in the $(n-1)$--skeleton of the complex. The result is a
(possibly infinite dimensional) simply connected, non-positively
curved cube complex, a cubing \cite[Theorem 3.7]{Sag95}.

There is an equivalence relation on the (directed) edges of a
cubing. Two directed edges $e$ and $e'$ are equivalent if there is a
finite sequence $e=e_0, e_1, \dots, e_k=e'$ such that for each $i$,
$e_i$ and $e_{i+1}$ are opposite edges of some 2--cube, oriented in the
same direction.

Equivalence classes of edges are called \emph{combinatorial
  hyperplanes}. There is a corresponding idea of a \emph{geometric
  hyperplane}.
Consider an $n$--cube of the complex. It can be identified with a cube
of side length 1 in $\mathbb{R}^n$ where the vertices have all
coordinates in $\{\pm\frac{1}{2}\}$. Consider the edges that correspond
to changing the $n$-th coordinate from $-\frac{1}{2}$ to
$\frac{1}{2}$. These edges belong to a combinatorial hyperplane. The
corresponding portion of a geometric hyperplane is the intersection of the $n$--cube
with the coordinate hyperplane $\{(x_1,\dots, x_n)\in\mathbb{R}^n\mid
x_n=0\}$. Such pieces are then glued together for each cube with edges
in the combinatorial hyperplane.

\begin{theorem}{\rm {\cite[Theorem 4.10]{Sag95}}}
  Suppose $J$ is a geometric hyperplane in a cubing $Y$. Then $J$ does
  not intersect itself and partitions $Y$ into two connected components.
\end{theorem}

We take the metric on the cubing to be the path metric on the
1-skeleton. The distance between two vertices is the minimal number of
edges in an edge path joining them, and such a minimal edge path is
called a \emph{geodesic}.

A corollary of the preceding theorem is the following observation
about geodesics:
Let $x$ and $y$ be vertices in a cubing $Y$. If they are distance $D$
apart, then a geodesic joining them must cross $D$ geometric
hyperplanes, one through the midpoint of each edge of the path. Each
of these hyperplanes disconnects $Y$, with $x$ and $y$ in opposite
components. Therefore, any geodesic from $x$ to $y$ must cross the
same $D$ hyperplanes. Conversely, the distance between $x$ and $y$ in
$Y$ is the number of hyperplanes separating them.

Fix a hyperplane. There is an $A_i^\e\in\Sigma$ such
that every directed edge $e$ in the hyperplane joins a vertex $V_e$ with a
vertex $(V_e\setminus\{A_i^\e\})\cup \{A_i^{1+\e}\}$. Furthermore,
every edge of this form belongs to the hyperplane
\cite[Lemma 3.9]{Sag95}.

Thus, we have a bijection between the set of geometric hyperplanes
and the collection $\{S_i\}$ of cut sets. This is how the cubing
encodes the collection of cut sets. Cut sets of $X$ correspond to
hyperplanes of $Y$. Distance in $Y$ corresponds to being separated by a
given number of cut sets. An $n$--cube in $Y$ corresponds to a
collection of $n$ distinct, pairwise crossing cut sets $S_i$ in $X$.

\subsection{Graphs and Complexes of Groups}
In this section we give a brief account of graphs and complexes of
groups. The reader is referred to Bridson and Haefliger's book
\cite{BriHae99}  for more detail.

A \emph{graph of groups} is a construction that builds a group by
amalgamating smaller groups. Start with a finite connected graph $\Gamma$, and
associate to each vertex or edge $\gamma$ a local  group $G_\gamma$, along with 
injections $\phi_{e,v}\from G_e\to G_v$ for each edge $e$ and vertex
$v$ that is an endpoint of $e$.

The \emph{fundamental group of the graph of groups} is then obtained by
taking as generators all the vertex groups as well as one generator
$g_e$ 
for each edge $e$ in the graph. The relations are:
\begin{enumerate}
\item all the relations from the vertex groups,
\item for each edge $e$ with endpoints $v$ and $v'$, and for each
  $h\in G_e$, \[g_e\phi_{e,v}(h)g_e^{-1}=\phi_{e,v'}(h),\]
\item $g_e=1$ for each edge $e$ in a chosen maximal subtree of $\Gamma$.
\end{enumerate}

The fundamental group does not depend on the choice of maximal
subtree.

Associated to a graph of groups there is a simplicial tree
$D\varGamma$ covering $\Gamma$ called the
Bass-Serre tree or the \emph{development} of the graph of groups. The
fundamental group of the graph of groups acts by isometries on
$D\varGamma$, with vertex stabilizers equal to conjugates of the
vertex groups in the graph of groups, and edge stabilizers equal to
conjugates of the edge groups.

Conversely, given a cocompact isometric action of a group $G$ on a
simplicial tree we get a graph of groups decomposition for $G$ by
taking the graph to be the quotient of the tree by the $G$ action and
choosing local groups to be vertex and edge stabilizers.

A \emph{complex of groups} is generalization of the graph of groups to
higher dimensional complexes. In particular, a group acting
cocompactly by isometries on a polyhedral complex can be given a
complex of groups structure by associating to each cell in the
quotient a group isomorphic to the stabilizer of the cell in the
original complex.

Unlike in the graph of groups case, not every complex of groups is
developable. That is, starting with a complex of groups $\Gamma$, there may not
exist a complex $X$ so that the fundamental group of the complex of
groups acts on $X$ with quotient $\Gamma$. However, if you start with
a group acting on a polyhedral complex, then the resulting graph of
groups is developable, the development is just the polyhedral complex
that you started with.

A developable complex of groups is \emph{faithful} if no non-trivial
element of the fundamental group of the complex of groups acts
trivially on the development.

To insure that the quotient is still a polyhedral complex, one should
assume that if an element of the group leaves a cell invariant, then
it fixes it pointwise. This is called an action \emph{without
 inversions}. If this is not the case, it can be achieved by
subdividing cells.

Lim and Thomas have worked out a covering theory for complexes of
groups \cite{LimTho08}. A particular result that will be of interest
to us is:

\begin{reftheorem}{\rm{\cite[Theorem 4]{LimTho08}}}
Let $X$ be a simply connected polyhedral complex, and let $G$ be a
subgroup of $\Aut(X)$ (acting without inversions) that induces a
complex of groups $\Gamma$. Then there is a bijection between the set of
subgroups of $\Aut(X)$ (acing without inversions) that contain $G$,
an the set of isomorphism classes of coverings of faithful,
developable complexes of groups by $\Gamma$.
\end{reftheorem}

If $G$ acts cocompactly on $X$ then so does any subgroup $H$ of $\Aut(X)$
containing $G$, and we get a covering of the compact quotient
complexes. If the complex of groups coming from the $H$ action has
finite local groups then we get finite covering, so $G$ is a finite index
subgroup of $H$.

\subsection{Coarse Geometry}

In this section and the next we establish the language and basic ideas of coarse geometry and trees. Again, see Bridson and Haefliger's book \cite{BriHae99}  for more detail.

Let $(X,d_X)$ and $(Y,d_Y)$ be metric spaces. Let $A$ and $B$ be subsets of $X$.

The (open) $r$--neighborhood of $A$ is the set $N_r(A)=\{x\in X\mid d_X(x,A)<r\}$.

The \emph{Hausdorff distance} between $A$ and $B$ is: \[\haus(A,B)=\inf \{r\mid A\subset \closure{N_r(B)} \text{ and } B\subset \closure{N_r(A)} \}\]

We will use the common convention that some object is $r$--\emph{[adjective]} if it has the property for the specified $r$, and is \emph{[adjective]} if there exists some $r$ such that the object is $r$--[adjective].

$A$ and $B$ are $r$--\emph{coarsely equivalent} if $\haus(A,B)\leq r$.

$A$ is $r$--\emph{coarsely dense} in $X$ if $A$ is $r$--coarsely equivalent to $X$.

A map $\phi\from X\to Y$ is a $(\lambda,\epsilon)$-\emph{quasi-isometric embedding} if there exist $\lambda \geq 1$ and $\epsilon\geq 0$ such that for all $x,x'\in X$:
\[\frac{1}{\lambda}d_X(x,x')-\epsilon\leq d_y(\phi(x),\phi(x'))\leq\lambda d_X(x,x')+\epsilon\]
If, in addition, the image of $\phi$ is $\epsilon$-coarsely dense in $Y$, then $\phi$ is a $(\lambda,\epsilon)$-\emph{quasi-isometry}.

Maps $\phi$ and $\psi$ from $X$ to $Y$ are $r$--\emph{coarsely equivalent}, or are equivalent \emph{up to $r$--bounded distance},  if for all $x \in X$, $d_Y(\phi(x),\psi(x))\leq r$.

$\QIgp(X,Y)$ is the set of quasi-isometries from $X$ to $Y$ modulo coarse equivalence.

Suppose $A$ is $r$--coarsely dense in $X$ and $\phi$ is a pseudo-map that assigns to each $a\in A$ a subset $\phi(a)$ in $Y$ of diameter at most $R$.
Suppose there are $\lambda\geq 1$ and $\epsilon\geq 0$ such that for all $a$ and $a'$ in $A$:
\[\frac{1}{\lambda} d_X(a,a')-\epsilon-R\leq \inf \{d_Y(y,y')\mid y\in\phi(a), y'\in\phi(a')\}\]
and
\[\sup\{d_Y(y,y')\mid y\in\phi(a), y'\in\phi(a')\}\leq \lambda d_x(a,a')+\epsilon+R\]

Then the pseudo-map $\phi$ determines a unique (up to coarse equivalence) extension to a $(\lambda,2\lambda r+\epsilon+R)$--quasi-isometric embedding $\Phi\from X\to Y$ such that for all $a\in A$, $\Phi(a)\in\phi(a)$. 
For each $x\in X$ choose a closest $a\in A$ and choose any $\Phi(x)\in \phi(a)$.

Suppose for some $x$ we define $\Phi'(x)$ by choosing a different closest $a'\in A$ and $\Phi'(x)\in\phi(a')$. Then
\begin{align*}
d_Y(\Phi(x),\Phi'(x))&\leq \sup\{d_Y(y,y')\mid y\in\phi(a), y'\in\phi(a')\}\\
&\leq \lambda d_X(a,a')+\epsilon+R\\
&\leq \lambda\cdot 2r+\epsilon+R,
\end{align*}
so $\Phi$ and $\Phi'$ are coarsely equivalent.

The fact that $\Phi$ is a quasi-isometric embedding follows easily.

If $\phi\from X\to Y$ is a $(\lambda,\epsilon)$ quasi-isometry, consider the inverse pseudo-map that takes a point in $\phi(X)$ to its preimage in $X$. 
This preimage has diameter at most $\epsilon$, and the image of $\phi$ is $\epsilon$-coarsely dense in $Y$.
We can therefore extend this pseudo-map to a $(\lambda, 2\epsilon(\lambda+\epsilon))$--quasi-isometry $\bar{\phi}\from Y\to X$. 
The compositions $\phi\circ \bar{\phi}$ and $\bar{\phi}\circ \phi$ are coarsely equivalent to the identity maps in $Y$ and $X$, respectively. 
We call $\bar{\phi}$ a \emph{coarse inverse} of $\phi$.

With this notion of inverse, the set $\QIgp(X)$ of quasi-isometries from $X$ to itself, modulo coarse equivalence, becomes a group, the quasi-isometry group of $X$.

Let $G$ be a finitely generated group and let $\mathcal{B}$ be a finite generating set. 
The \emph{word metric} on $G$ with respect to $\mathcal{B}$ is defined by setting $|g|$ to be the minimum length of a word equal to $g$ in $G$ written in terms of generators in $\mathcal{B}$ or their inverses.

The Cayley graph of $G$ with respect to $\mathcal{B}$ is the graph with one vertex for each element of $G$ and an edge $[g,g']$ connecting vertex $g$ to vertex $g'$ if $g'=gb$ for some $b\in\mathcal{B}$.
Make this a metric graph by assuming that each edge has length one. The distance between two vertices $g$ and $g'$ is the length of the shortest edge path joining them. Thus, the distance from $g$ to the identity vertex is the same as $|g|$ in the word metric.
$G$ acts on the Cayley graph by isometries via left multiplication.

While the Cayley graph depends on the choice of finite generating set, different choices yield quasi-isometric graphs. 
More generally, if $G$ acts properly and cocompactly by isometries on a length space $X$, then $X$ is quasi-isometric to $G$ with (any) word metric.
We call such a space $X$ a \emph{geometric model} of $G$.

\subsection{Free Groups and Trees}

Let $F$ be the free group of rank
$n$, with free generating set (free basis) $\mathcal{B}=\{a_1,\dots , a_n\}$. 
For $g\in F$, let $\bar{g}$ denote $g^{-1}$.

Let $\tree=\mathcal{C}_\mathcal{B}(F)$ be the Cayley graph of $F$ with
respect to $\mathcal{B}$. 
Since we have chosen a free generating set, $\tree$ is a tree, a graph with no loops.

The tree has a boundary at infinity $\bdry \tree$ that is a Cantor
set. 
Adding the boundary compactifies the tree; $\closure{\tree}=\tree\cup\bdry\tree$ is a compact topological space
whose topology agrees with the metric topology on $\tree$.
For any two points $t$ and $t'$ in $\closure{\tree}$ there is a unique
geodesic $[t,t']$ joining them.

Let $v$ and $w$ be vertices in $\tree$. Define:
 \[\shadow^v(w)=\{x\in\closure{\tree}\mid w\in[v,x]\}\]

Let $\shadow_\infty^v(w)=\shadow^v(w)\cap \bdry\tree$.

If $\xi\in\bdry\tree$ and $v\in\tree$ let $v=v_0,v_1,\dots$ be the vertices along $[v,\xi]$. The sets $\shadow^v(v_i)$ give a neighborhood basis for $\xi$.
 The topology on $\closure{\tree}$ is independent of the choice of $v$.

Since $\tree$ is hyperbolic, any
quasi-isometry $\phi\from \tree\to \tree'$ extends to a homeomorphism $\bdry
\phi\from \bdry \tree\to\bdry \tree'$.


\subsection{Line Patterns and the Decomposition Space}
Suppose $l=\{gw^m\}_{m\in \mathbb{Z}}$ is a line in the pattern. The
line $l$ has
distinct endpoints at infinity:
\[l^+=gw^\infty=\lim_{i\to \infty}gw^i\]
and
\[l^-=gw^{-\infty}=\lim_{i\to -\infty}gw^i\]

The lines in the pattern never have endpoints in common, so we can
decompose $\bdry \tree$ into disjoint subsets that are either the pair of
endpoints of a line from the pattern or a boundary point that is not
the endpoint of a line.

Define the decomposition space $\decomp_\lines$ (or just
$\decomp$ when $\lines$ is understood) associated to a line
pattern $\lines$ to be the space that has one point for each set
in the decomposition of $\bdry \tree$, with the quotient topology.

Let $q\from \bdry \tree \to \decomp$ be the quotient map. 
For $x\in\decomp$, $q^{-1}(x)$ is either a single point that is
not the endpoint of any line in $\lines$, or $q^{-1}(x)=\{l^+,l^-\}$
for some $l\in\lines$.
The former we call \emph{bad points}, the later, \emph{good points}.

The quotient map $q$ induces a bijection between $\lines$ and the
good points of $\decomp$, which we denote by $q_*$.

If $S\subset \decomp$ we will use the notation
$\hat{S}=q^{-1}(S)\subset \bdry\tree$. Further, if $S$ consists of good points
we will use $\tilde{S}$ to be the collection of lines of $\lines$
given by $q_*^{-1}(S)$.

The decomposition space  is a perfect, compact, Hausdorff topological space.

A quasi-isometry $\phi$ from $\tree$ to $\tree'$ extends to a homeomorphism
$\bdry\phi\from \bdry \tree\to\bdry \tree'$.
In particular, if there are line patterns $\lines$ in $\tree$ and
$\lines'\in \tree'$, and if $\phi$ is a pattern preserving quasi-isometry,
then the homeomorphism $\bdry\phi\from\bdry \tree\to\bdry \tree'$ descends to a
homeomorphism of the corresponding decomposition spaces.

\section{Whitehead's Algorithm}\label{sec:WA}
Since Whitehead's original  work \cite{Whi36a,Whi36}, a number of authors have refined Whitehead's Algorithm and applied it
to related algebraic questions. Section I.4 of the book of Lyndon and
Schupp \cite{LynSch77} gives a version of Whitehead's Algorithm and
some of the classical applications.

More recently, Stallings \cite{Sta99} and Stong \cite{Sto97} gave
3--manifold versions of Whitehead's Algorithm. In each of these papers
the aim was to show that a version of Whitehead's Algorithm could be
used to determine if, given a finite list of words $(w_1,\dots w_k)$
in $F$, there is a free splitting of $F$ such that every $w_i$ is
conjugate into one of the free factors. Stallings calls this
``algebraically separable''. This algebraic question is
then shown to be equivalent to a geometric question about whether or
not a collection of curves in a handlebody has a property that
Stallings calls ``geometrically separable'' and Stong calls ``disk-busting''.

In this section we
review Whitehead's Algorithm. Our language is similar to that of
Stallings and Stong, except that our group actions are on the left and
path concatenations are on the right, while they use the opposite convention.

\subsection{Whitehead Graphs}
Let $w\in F$ be a cyclically reduced word. Let
$\mathcal{B}=\{a_1,\dots, a_n\}$ be a free basis of $F$.
The \emph{Whitehead Graph of $w$ with respect to $\mathcal{B}$}, $\Wh_\mathcal{B}(*)\{w\}$, is the graph
with $2n$ vertices labeled $a_1,\dots, a_n, \bar{a}_1,\dots \bar{a}_n$, and an edge between vertices $v$ and $v'$ for \emph{each}
occurrence of $\bar{v}v'$ in $w$ (as a cyclic word).
The graph depends on the choice of $\mathcal{B}$, and, of course, on
$w$, but we will write $\Wh(*)$ when these are clear.

\begin{remark}
  At present the $(*)$ may be ignored; it will be explained in the next section.
\end{remark}

For example, if $F=\left<a,b\right>$ Figures~\ref{fig:Wh(a)}-\ref{fig:1Wh(a2ba2B2)} show some Whitehead Graphs.

\begin{figure}[ht]
\begin{minipage}[b]{0.45\textwidth}
\labellist \small
\pinlabel $b$ [bl] at 87 164
\pinlabel $\bar{b}$ [tl] at 87 3
\pinlabel $a$ [tl] at 165 83
\pinlabel $\bar{a}$ [tr] at 4 83
\endlabellist
  \centering
  \includegraphics[height=\figstandardheight]{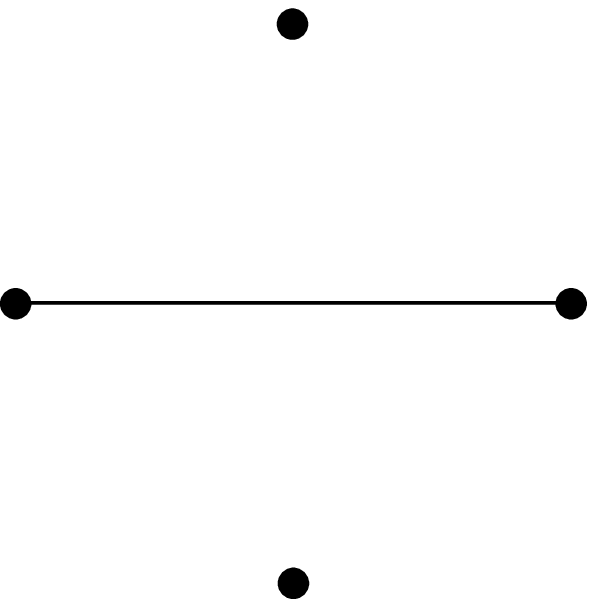}
  \caption{  $\Wh(*)\{a\}$}
\label{fig:Wh(a)}
\end{minipage}
\hfill
\begin{minipage}[b]{0.45\textwidth}
\labellist \small
\pinlabel $b$ [bl] at 87 164
\pinlabel $\bar{b}$ [tl] at 87 3
\pinlabel $a$ [tl] at 165 83
\pinlabel $\bar{a}$ [tr] at 4 83
\endlabellist
  \centering
  \includegraphics[height=\figstandardheight]{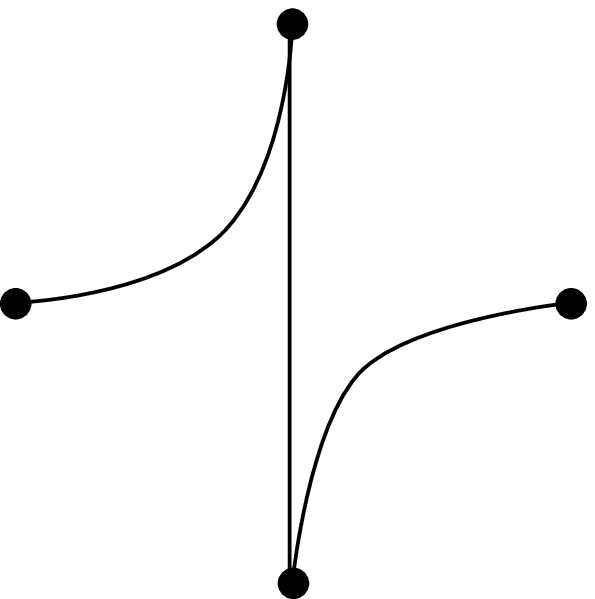}
  \label{fig:ab2}
  \caption{  $\Wh(*)\{ab^2\}$}
\end{minipage}

\bigskip \bigskip
\begin{minipage}[b]{0.45\textwidth}
\labellist \small
\pinlabel $b$ [bl] at 87 164
\pinlabel $\bar{b}$ [tl] at 87 3
\pinlabel $a$ [tl] at 165 83
\pinlabel $\bar{a}$ [tr] at 4 83
\endlabellist
  \centering
    \includegraphics[height=\figstandardheight]{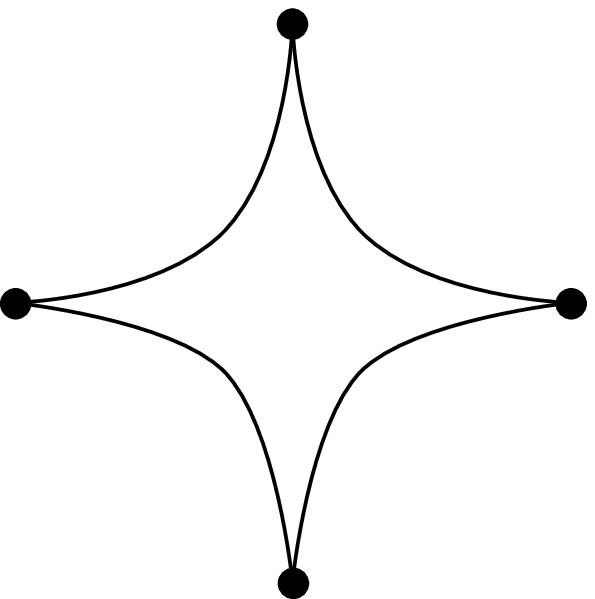}
  \caption{  $\Wh(*)\{ab\bar{a}\bar{b}\}$}
  \label{fig:Wh(abAB)}
\end{minipage}
\hfill
\begin{minipage}[b]{0.45\textwidth}
\labellist \small
\pinlabel $b$ [bl] at 87 164
\pinlabel $\bar{b}$ [tl] at 87 3
\pinlabel $a$ [tl] at 165 83
\pinlabel $\bar{a}$ [tr] at 4 83
\endlabellist
  \centering
\includegraphics[height=\figstandardheight]{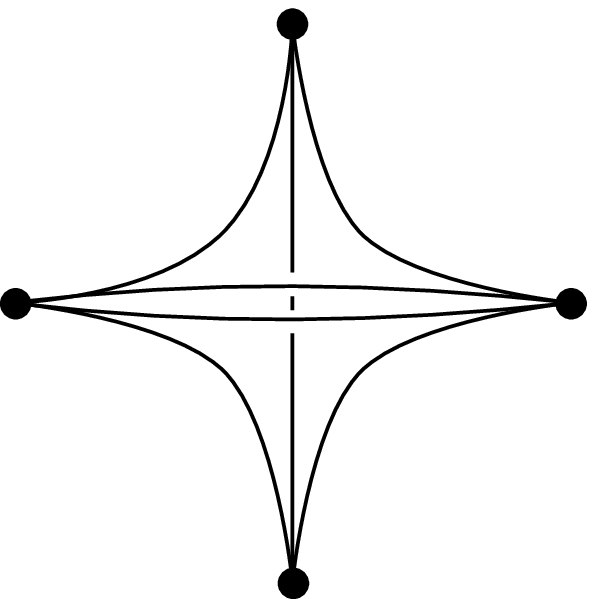}
  \caption{$\Wh(*)\{a^2ba^2\bar{b}^2\}$}
  \label{fig:1Wh(a2ba2B2)}
\end{minipage}
\end{figure}

Notice that $\Wh(*)\{a\}$ is disconnected; the vertices $b$ and $\bar{b}$
are isolated. 

$\Wh(*)\{ab^2\}$ is connected but becomes disconnect if
the vertex $b$ (or $\bar{b}$) is deleted;
$b$ and $\bar{b}$ are \emph{cut vertices}.

$\Wh(*)\{ab\bar{a}\bar{b}\}$ is connected and has no cut
vertices.

$\Wh(*)\{a^2ba^2\bar{b}^2\}$ is also connected with no cut vertices,
and has multiple edges between vertices $a$ and $\bar{a}$.

More generally, one can make a Whitehead graph representing  finitely
many words $w_1,\dots, w_m$.
We call this Whitehead graph $\Wh(*)\{w_1,\dots ,w_m\}$ or  
$\Wh(*)\{\mathcal{L}\}$, where $\mathcal{L}$ is the line pattern generated by
$\{w_1,\dots , w_m\}$.

\subsection{Whitehead Automorphisms}
Let $\phi$ be an automorphism of $F$. 
Applying $\phi$ changes the Whitehead graph
$\Wh_\mathcal{B}(*)\{w_1,\dots , w_m\}$ to the Whitehead graph
$\Wh_\mathcal {B}(*)\{[\phi(w_1)],\dots , [\phi(w_m)]\}$, where  $[\phi(w_i)]$
means choose a cyclically reduced word in the conjugacy class of $\phi(w_i)$.

An automorphism that permutes $\mathcal{B}$ or swaps a generator with
its inverse gives an isomorphic Whitehead graph.

\begin{definition}\label{definition:whaut}
A \emph{Whitehead automorphism} is an automorphism of the following
form:
Pick $x\in \mathcal{B}\cup\bar{\mathcal{B}}$, a generator or the
inverse of a generator.
Pick $Z\subset \mathcal{B}\cup\bar{\mathcal{B}}$ such that $x\in Z$
and $\bar{x}\notin Z$.

  Define an automorphism $\phi_{x,Z}$ by defining $\phi_{x,Z}(x)=x$ and for
the rest of the generators $y\in\mathcal{B}$:
\[\phi_{x,Z}(y)=
\begin{cases}
  y & \text{ if } y\notin Z\text{ and }\bar{y}\notin Z\\
  xy & \text{ if } y\in Z\text{ and }\bar{y}\notin Z\\
  y\bar{x} & \text{ if } y\notin Z\text{ and }\bar{y}\in Z\\
  xy\bar{x} & \text{ if } y\in Z\text{ and }\bar{y}\in Z\\
\end{cases}
\]
\end{definition}

We say that the automorphism $\phi_{x,Z}$ is the \emph{Whitehead
automorphism that pushes $Z$ through $x$}.

To visualize what is happening, consider the rose with one vertex and
one oriented loop for each element of $\mathcal{B}$. The fundamental group is $F$.
The Whitehead automorphism $\phi_{x,Z}$ is the automorphism of the
fundamental group induced by the homotopy equivalence that pushes 
one or both ends of the $y$--loop around the $x$--loop according to whether
$y$ or $\bar{y}$ or both are in $Z$, or leaves the $y$--loop alone if
neither $y$ nor $\bar{y}$ are in $Z$. See also \fullref{sec:whautrevisit}.

Define the \emph{complexity} of the collection $w_1,\dots,
w_n$ to be the number of edges of $\Wh(*)\{w_1,\dots, w_m\}$.   
This is equivalent to the sum of the lengths of the $w_i$, and also  half the sum of the valences of the
vertices. 

Comparing $\Wh_\mathcal{B}(*)\{w_1,\dots , w_m\}$ to
$\Wh_\mathcal {B}(*)\{[\phi_{x,Z}(w_1)],\dots , [\phi_{x,Z}(w_m)]\}$
we see that the valences of vertices other than $x$ and $\bar{x}$ do
not change. 
The new valence of $x$ and $\bar{x}$ is equal to the number of edges
that go between $Z$ and $Z^c$.
Thus, the Whitehead automorphism reduces the complexity of the
Whitehead graph exactly when there are fewer
edges joining $Z$ and $Z^c$ than the valence of $x$.

\begin{reftheorem}
  $\Aut(F)$ is generated by:
  \begin{enumerate}
  \item exchanges of a generator with its inverse
\item permutations of the generators
\item Whitehead automorphisms
  \end{enumerate}
\end{reftheorem}
  This is clear since this set of automorphisms contains the Nielsen
  generators for $\Aut(F)$ \cite{Nie24}.

Whitehead's Algorithm is as follows: 
First, check if any Whitehead automorphisms reduce
the complexity of the Whitehead graph. Repeat. Once you have reduced
to minimal complexity, there are only finitely many graphs to
consider. Build a graph with one vertex for each possible Whitehead
graph with the given complexity, and an edge between two vertices if
one of the given generators of the automorphism groups takes one graph
to the other. One can then show that the desired automorphism exists
if and only if the reduced Whitehead graphs for the two lists of words
lie in the same connected component of this graph.

If $\{w_1,\dots , w_m\}$ is a subset of a free basis then the minimal
complexity Whitehead graph should have  $m$ disjoint edges.

If there is a free splitting $F=F'*F''$ with every $w_i$ in $F'$ or
$F''$ then the minimal complexity Whitehead graph should be disconnected. 

The presence of a cut vertex in the Whitehead graph indicates that the
graph is not reduced. If $x$ is a cut vertex, let $Z$ be the union of
$\{x\}$ and the vertices
of a connected component
of $\Wh(*)\{\mathcal{L}\}\setminus \{x\}$ not containing $\bar{x}$.
The Whitehead automorphism $\phi_{x,Z}$ reduces complexity.

One application of Whitehead's Algorithm is that a word $w$ is an
element of a free basis of $F=F_n$ if and only if the minimal complexity
Whitehead graph for $w$ consists of a single edge and $2(n-1)$
isolated vertices.

More generally, the \emph{width} of an element $w$ is the rank of the
smallest free factor of $F$ containing $w$. The minimal complexity
Whitehead graph for an element of width $m$ in $F=F_n$ consists of
$2(n-m)$ isolated vertices and a connected graph without cut vertices
on the remaining vertices.

\section{Whitehead Graphs and the Topology of the Decomposition Space}\label{sec:Whandtopology}

The decomposition space associated to a line pattern first appears
in the literature in work of Otal \cite{Ota92}, who proves that the
decomposition space is connected if and only if there exists a basis
$\mathcal{B}$ of $F$ such that $\Wh_\mathcal{B}(*)$ is connected
without cut vertices.

A similar theorem appears in the thesis of
Reiner Martin \cite{Mar95}, who references notes of Bestvina.

\begin{theorem}{\rm {\cite[Theorem 49]{Mar95}}}\label{theorem:Bestvina}
  For any $w\in F-\{1\}$, the following are equivalent:
  \begin{enumerate}
  \item $w$ is contained in a proper free factor of $F$.
\item The width of $w$ is strictly less than the rank of $F$.
\item There exists a disconnected Whitehead graph of $w$.
\item The decomposition space associated to the pattern generated by
  $w$ is disconnected.
\item Every Whitehead graph for $w$ with no cut vertices is
  disconnected.
  \end{enumerate}
\end{theorem}

The goal of this section is to further explore
the relationship between generalizations of the Whitehead graph and
the topology of the decomposition space. 
In particular, we are interested in finite cut sets in the case that the
decomposition space is connected.

\begin{remark}
  The theorem stated in \cite{Mar95} has an additional equivalent
  condition: for any basis there exists a generalized Whitehead graph
  that is disconnected. We will not make use of this. In our notation,
  Martin's generalized Whitehead graph is $\Wh_\mathcal{B}(N_r(*))\{w\}$ (see below).
\end{remark}

\subsection{Geometric Interpretation and Generalizations}
Fix a free basis $\mathcal{B}$ for $F_n$, and let $\tree$ be the
corresponding Cayley graph. 

Let $\X$ be a closed, connected subset of $\closure{\tree}$.
Consider the connected components of $\closure{\tree}\setminus \X$.
Take these components as the vertices of a graph. 
Connect vertices $v_1,\,v_2$ by an edge  if there is a line in $l\in\mathcal{L}$ with
one endpoint in the component corresponding to $v_1$ and the other in
the component corresponding to $v_2$.
Call this graph $\Wh_\mathcal{B}(\X)\{\mathcal{L}\}$, and notice that when $\X=*$ is a
single vertex
this graph is exactly $\Wh_\mathcal{B}(*)\{\mathcal{L}\}$.
Since $\mathcal{L}$ is equivariant we get the same graph for any
choice of vertex.

  We will also give a combinatorial construction of our generalization of
  the Whitehead graph. However, the intuition that informs our
  arguments comes from the above geometric interpretation. 
One should visualize the Whitehead graph as the portion of the
line pattern that passes through a subset of a tree, rather than as an
abstract graph. Where appropriate, as in \fullref{fig:abaBloose} and
\fullref{fig:abouttosplice}, we have included the relevant portions of the
tree to aid in this visualization.

If $\X$ is a finite connected subset of $\tree$ we can build up
Whitehead graphs $\Wh(\X)\{\mathcal{L}\}$ in a combinatorial way
by \emph{splicing} together copies of $\Wh(*)\{\mathcal{L}\}$ for each of the vertices of
$\X$.
Splicing is a method of combining graphs. 
The term was coined by
Manning in \cite{Man09} where he uses splicing to construct Whitehead
graphs of finite covers of a handlebody from the Whitehead graph of
the base handlebody.

Let $v$ be a vertex of a graph $\Gamma$ and let $v'$ be a vertex of a
graph $\Gamma'$ of valence equal to the valence of $v$.
Given a bijection between edges of $\Gamma$ incident to $v$ and
edges of $\Gamma'$ incident to $v'$, the \emph{splicing map}, form a
new graph whose vertices are the vertices of $\Gamma$ and $\Gamma'$
minus the vertices
$v$ and $v'$. 
Edges not incident to $v$ or $v'$ are retained in the new
graph. Finally, for each pair of edges $[w,v]$ in $\Gamma$ and
$[w',v']$ in $\Gamma'$ identified by the splicing map, add an edge
$[w,w']$ in the new graph.

In other words, we have deleted $v$ and $v'$, leaving the edges
incident to those vertices with ``loose ends''. The splicing map tells
us how to splice a loose end at $v$ to  a loose end at $v'$ to get an
edge in the new graph.

For Whitehead graphs the splicing map is determined by the line
pattern.
Suppose we have adjacent vertices $g$ and $ga$ in a Cayley graph, and
corresponding Whitehead graphs $\Wh(g)$ and $\Wh(ga)$.
We splice them together to build the Whitehead graph
$\Wh([g,ga])$. The $g$ vertex and $ga$ vertex in $\tree$ are adjacent
across an $a$--edge, so the splicing vertices are the $a$--vertex of
$\Wh(g)$ and the $\bar{a}$--vertex of $\Wh(ga)$. Each edge in $\Wh(g)$
incident to $a$ corresponds to a length two subword of one of the generators of
the line pattern of the form $xa$ or $\bar{a}x$.
Suppose an edge corresponds to a subword $xa$, and suppose the next
letter is $y$, so there is a length three subword $xay$. 
We define the splicing map to identify the edge corresponding to this
particular instance of the
subword $xa$ to the edge in $\Wh(ga)$ (incident to $\bar{a}$) corresponding to this particular
instance of the subword $ay$.

We can make the splicing easier to visualize if we draw the Whitehead graphs with
loose ends at the vertices. 
\fullref{fig:abaBloose} shows the Whitehead graph for the pattern
generated by the words $ab$ and $a\bar{b}$ in
$F=F_2=\left<a,b\right>$, along with the underlying tree.
The word $ab$ will
contribute an edge from $\bar{a}$ to $b$ and an edge from $\bar{b}$ to
$a$. The twists in the graph indicate the splicing maps.

\begin{figure}[h]
\labellist \small
\pinlabel $*$ [tr] at 92 89
\pinlabel $a$ [t] at 180 90
\pinlabel $\bar{a}$ [t] at 5 90
\pinlabel $b$ [r] at 91 179
\pinlabel $\bar{b}$ [r] at 91 3
\endlabellist
  \centering
  \includegraphics[height=1.5\figstandardheight]{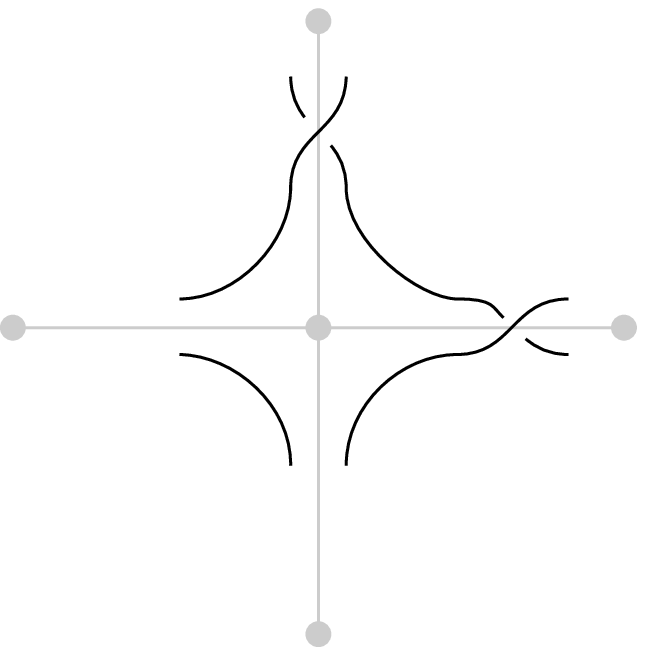}
  \caption{$\Wh(*)\{ab, a\bar{b}\}$}
  \label{fig:abaBloose}
\end{figure}

Let $*$ be the identity vertex.
Take a copies of this graph at $*$ and at $a$ and splice them
together. We get the splicing map by considering the words. 
There is an $ab$--line at $*$. If the first letter is $a$, the previous
letter was $b$, so we see an edge from $\bar{b}$ to $a$ in the
Whitehead graph at $*$.
The next letter is $b$, so in the Whitehead graph at $a$ we see an
edge from $\bar{a}$ to $b$, and the twist in the graph indicates that
these two edges should be spliced together.

Similarly, there is an $a\bar{b}$--line at $*$. It contributes an edge
from $b$ to $a$ in the Whitehead graph at $*$, and this continues on
to an edge from $\bar{a}$ to $\bar{b}$ in the Whitehead graph at $a$.

\begin{note}
  Unless noted otherwise, figures are drawn so that the splicing map is achieved by an orientation preserving isometry of the page.
\end{note}
\begin{figure}[h!]
\labellist \small
\pinlabel $*$ [tr] at 91 89
\pinlabel $a$ [tr] at 212 90
\pinlabel $\bar{a}$ [tr] at 2 90
\pinlabel $b$ [r] at 90 179
\pinlabel $\bar{b}$ [r] at 90 3
\pinlabel $ab$ [r] at 211 179
\pinlabel $a\bar{b}$ [r] at 211 3
\pinlabel $a^2$ [lt] at 299 93
\endlabellist
  \centering
  \includegraphics[height=1.5\figstandardheight]{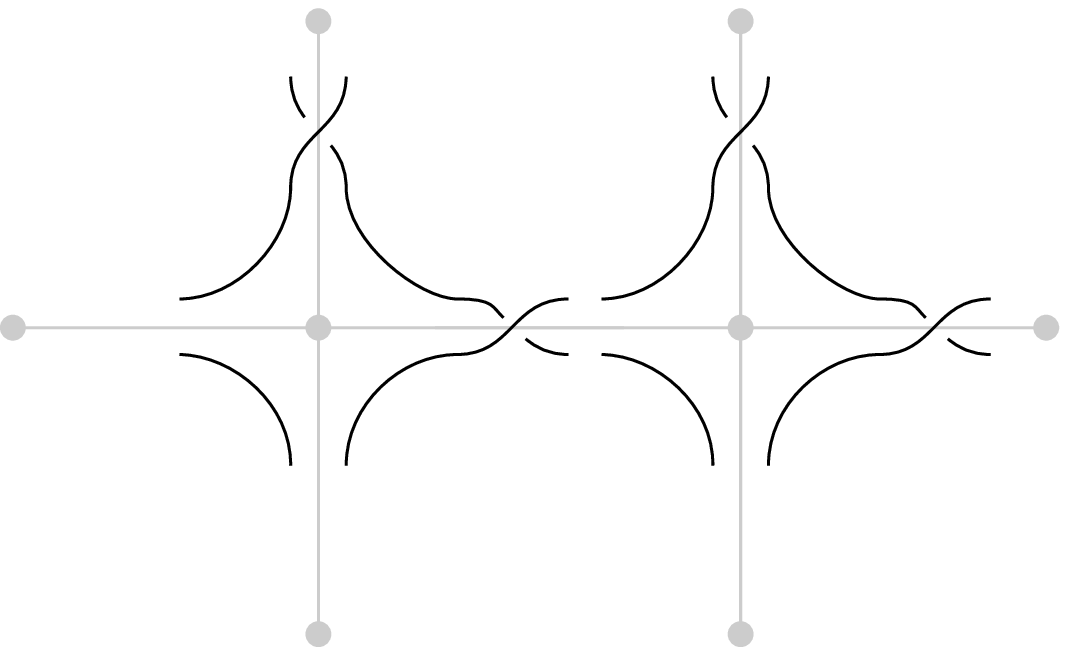}
  \caption{$\Wh(*)\{ab,a\bar{b}\}$ and $\Wh(\{a\})\{ab,a\bar{b}\}$}
\label{fig:abouttosplice}
\end{figure}

\begin{figure}[h!]
  \centering
  \includegraphics[height=\figstandardheight]{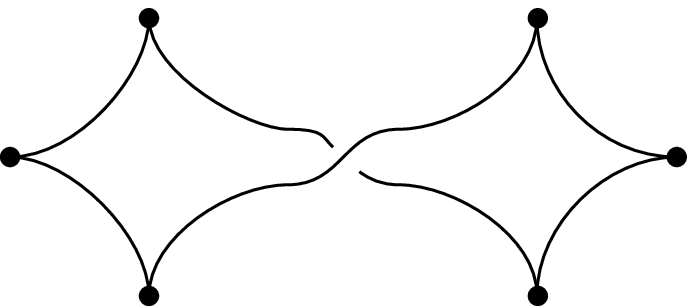}
  \caption{$\Wh([*,a])\{ab,a\bar{b}\}$}
\end{figure}

The geometric and splicing constructions produce the same graph for
sets $\X$ with finitely many vertices. 
We could try to take limits of the spliced graphs when $\X$ is
infinite, but
if $\X\subset\closure{\tree}$ contains endpoints of some $l\in\mathcal{L}$, then splicing does not actually produce a graph.

If both endpoints of $l$ are in $\closure{\X}$ then after finitely many
splices there is an edge corresponding to $l$, but in the limit the
edge grows to be an open interval not incident to any vertices; the vertices escape to infinity.
This line does not occur if we follow the geometric definition, because it is not joining two different components of the complement of $\X$ in $\closure{\tree}$.
Similarly, if only one endpoint of $l$ is in $\closure{\X}$ then splicing produces a graph $\mathcal{G}$ with a half line attached. 
If we throw out these ``non-closed edges'' we get the graph $\Wh(\X)\{\mathcal{L}\}$ of the geometric definition.

\begin{remark}
  Stong \cite{Sto97} defines a generalized Whitehead graph that
  coincides with our definition, but, like Martin \cite{Mar95}, only
  makes use of $\Wh_\mathcal{B}(N_r(*))\{w\}$.
\end{remark}

\subsection{Whitehead Automorphisms Revisited}\label{sec:whautrevisit}
Let us consider how application of a Whitehead automorphism changes
the line pattern. 

Suppose $x$, $y$ and $z$ are in $\mathcal{B}\cup\bar{\mathcal{B}}$
with $y\neq z$. Consider a Whitehead automorphism
$\phi=\phi_{x,Z}$ (recall \fullref{definition:whaut}).
Let $l\in\lines$ be a line that goes through vertices $y$, $*$ and
$z$, where $*$ is the identity vertex.
The line $l$ is the geodesic that goes through vertices of the form
$\{y(\bar{y}zu)^m\}_{m\in\mathbb{Z}}$, where $u$ is some word in $F$
that does not begin with $\bar{z}$ or end with $y$.

First suppose that $y,z\in Z$, $y,z\neq x$ and $\bar{y},\bar{z}\notin
Z$, so that $\phi(x)=x$, $\phi(y)=xy$ and $\phi(z)=xz$. Then $\phi(l)$
is the line that includes vertices of the form:
\[\{\phi(y(\bar{y}zu)^m)\}_{m\in\mathbb{Z}}=\{xy(\bar{y}\bar{x}xz\phi(u))^m\}_{m\in\mathbb{Z}}=\{xy(\bar{y}z\phi(u))^m\}_{m\in\mathbb{Z}}\]
Since $u$ does not begin with $\bar{z}$ or end in $y$, the same is
true for $\phi(u)$. Therefore, $\phi(l)$ goes through vertices $xy$,
$x$ and $xz$.
The line $l$ that went through $*$ has been ``pushed through'' the $x$
edge to a line $\phi(l)$ that goes through $x$ and not
through $*$.

Using similar arguments one can show:
\begin{enumerate}
\item $\phi(l)$ goes through $x$
and not through $*$ if $y$ and $z$ are in $Z$.
\item $\phi(l)$ goes through $*$ and not through $x$ if $y$ and $z$
  are in $Z^c$.
\item $\phi(l)$ goes through both $*$ and $x$ if exactly one of $y$ or
  $z$ is in $Z$.
\end{enumerate}

\subsection{Cut Sets in the Decomposition Space}
The next two lemmas use essentially the same ideas that go into the
proofs of \fullref{theorem:Bestvina} in \cite{Mar95} and
\cite[Proposition 2.1]{Ota92}.

\begin{lemma}\label{lemma:disconnected}
 If for some choice of basis $\Wh(*)$ is disconnected, then $\decomp$ is disconnected.
\end{lemma}
\begin{proof}
  Let $*$ be the identity vertex in $\tree$. Let $\mathcal{B}$ be a free
  basis of $F$ such that $\Wh_\mathcal{B}(*)$ is not
  connected. Vertices of $\Wh(*)$ are in bijection with
  $\mathcal{B}\cup\bar{\mathcal{B}}$.
There is some partition of
  $\mathcal{B}\cup\bar{\mathcal{B}}$ into subsets $\mathcal{A}$ and
  $\mathcal{A}'$ so that no lines of $\lines$ connect
  $\mathcal{A}$ to $\mathcal{A}'$.
Let \[\hat{A}=\cup_{a\in \mathcal{A}}\shadow_\infty^*(a)\]
The sets $\hat{A}$ and $\hat{A}^c\subset\bdry \tree$ are both nonempty
clopens, sets that are both open and closed.
Since there are no lines of $\lines$ with one endpoint in
$\hat{A}$ and one in $\hat{A}^c$, their images in $\decomp$ are
disjoint nonempty clopens, so $\decomp$ is disconnected.
\end{proof}

\begin{lemma}\label{lemma:basicconnectivity}
Suppose there exists a free basis $\mathcal{B}$ of $F$ such that
$\Wh_\mathcal{B}(*)$ is connected without cut vertices. 
Let
$\tree$ be the Cayley graph of $F$ corresponding to $\mathcal{B}$.
Pick any edge $e$ in $\tree$.
Let $*$ and $v$ be the endpoints of $e$.
Let $\hat{A}=\shadow_\infty^*(v)$. That is, $\hat{A}$ is the ``half''
of $\bdry \tree$ on the ``$v$-side'' of $e$.
The set $A=q(\hat{A})$ is connected in $\decomp$.
\end{lemma}

\begin{proof}
  \begin{figure}[h]
\labellist \small
\pinlabel $\hat{A}$ [l] at 169 59
\pinlabel $e$ [t] at 78 58
\pinlabel $*$ [tl] at 35 58
\pinlabel $v$ [tr] at 126 58
\endlabellist

    \centering
    \includegraphics[width=.4\textwidth]{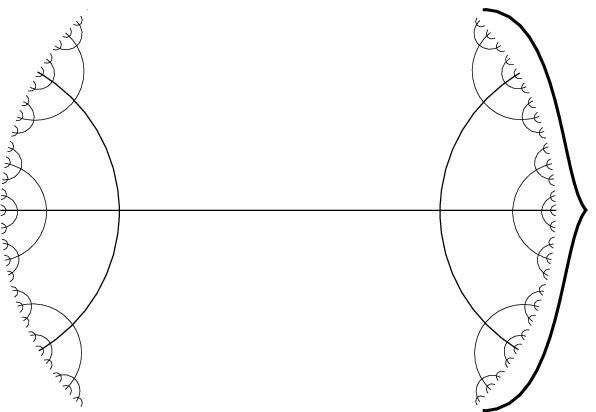}
    \caption{The boundary of the tree split into ``halves''.}
    \label{fig:halfboundary}
  \end{figure}

  Suppose  there are open sets $B$ and $C$ of $\decomp$ such that $A\subset B\cup C$ and $A\cap B\cap C=\emptyset$.
The set $\hat{A}$ is open in $\bdry \tree$, so $A'=\hat{A}\cap q^{-1}(B)$ and $A''=\hat{A}\cap q^{-1}(C)$ are open.
Assuming that $A'$ is nonempty, we will show that $A''$ must be empty, which implies $A$ is connected.
 
$\hat{A}$ is closed in $\bdry \tree$, so $A'$ and $A''$ are closed. 
Compactness of $\bdry \tree$ implies $A'$ and $A''$ are compact clopens. 
Since $A'$ is compact and open, there are finitely many vertices
$x_1,\dots ,x_a$ so that $A'=\cup_{i=1}^a\shadow_\infty^*(x_i)$

There is a similar finite collection $y_1,\dots , y_b$ that determines $A''$. 

Consider the convex hull $\hull$ of $\{x_i\}_{i=1}^a\cup\{y_j\}_{j=1}^b\cup
\{v\}$; it is a finite tree.
Call the vertices of $\hull$ other than $\{x_i\}_{i=1}^a\cup\{y_j\}_{j=1}^b\cup
\{v\}$ the ``interior vertices''.
Since $A'\cup A''=A$, $\hull$ includes all
edges incident to its interior vertices. 
Let $\X$ be the union of the set of interior vertices with $\{v\}$.
 
Construct the Whitehead graph 
$\Wh(\X)$.
It has $a+b+1$ vertices corresponding to the $x_i$ and $y_j$ and the
edge $e$ of $\tree$.

The graph is connected without cut points, since it can be
constructed by splicing together finitely many copies of
$\Wh(*)$, which is connected without cut points. In
particular, the vertex $e$ is not a cut vertex. 

Assume $A'$ is nonempty. If $x_1=v$ then $A'=A$, so $A''=\emptyset$,
and we are done. Otherwise, $x_1$ is a vertex of
$\Wh(\X)\setminus v$. An edge of $\Wh(\X)\setminus v$
incident to $x_1$ corresponds to a line $l\in\lines$ with one
endpoint in the shadow of $x_1$ and the other endpoint in the shadow
of $z$ for some $z\in \{x_i\}_{i=2}^a\cup\{y_j\}_{j=1}^b$.
In the decomposition space these two endpoints are identified, and we
already know that the image of the first endpoint is in $B$. This
means that $z$ must be in $\{x_2,\dots ,x_a\}$. 
Since $\Wh(\X)\setminus v$ is connected we conclude that all the vertices of $\Wh(\X)\setminus v$ belong to $\{x_1,\dots ,x_n\}$, so $A''=\emptyset$. Thus, $A$ is connected in $\decomp$.
\end{proof}

Thus, if $\Wh(*)$ is connected without cut vertices, then for any edge $e$ in $\tree$ the boundaries of
the two connected components of $\tree\setminus e$ correspond to connected
sets in the decomposition space. Since $\Wh(*)$ is connected
there is also at least one line in $\lines$ crossing $e$. This
means that these two connected sets in the decomposition space have a
point in common.

\begin{corollary}\label{corollary:connected}
Suppose $\Wh(*)$ has no cut vertices. Then the
decomposition space is connected if and only if $\Wh(*)$ is connected.
\end{corollary}

\begin{proof}[Proof of \fullref{theorem:Bestvina}]
Conditions (1) and (2) are equivalent by definition.
The equivalence of (1) and (3) is
a consequence of Whitehead's
Algorithm. 

(3)$\implies$(4) is \fullref{lemma:disconnected}. 

\fullref{corollary:connected} implies the contrapositive of (4)$\implies$(5).

It is always possible to eliminate cut vertices with Whitehead
automorphisms, so (5)$\implies$(3).
\end{proof}

Here is another corollary of \fullref{lemma:basicconnectivity}:
\begin{corollary}\label{corollary:edgeset}
 Suppose $\Wh(*)$ is connected without cut vertices. Pick any
 edge $e$ in $\tree$. Let $\tilde{S}$ be the collection of (finitely many)
 lines of $\lines$ that cross $e$. Then $S=q_*(\tilde{S})$ is a cut set in $\decomp$.
\end{corollary}

We will call such a set $S$ coming from all the lines crossing an edge
an \emph{edge cut set}.

From now on, unless otherwise noted, we will assume that any Whitehead
graph $\Wh(*)$ is connected without cut vertices. Thus, the
decomposition space is connected.
Our goal is to identify finite minimal cut sets.

We have an easy sufficient condition to see that a set $\tilde{S}=\{l_1,\dots ,
l_k\}\subset\lines$ gives a cut set $S=q_*(\tilde{S})$ of $\decomp$.
\begin{proposition}\label{proposition:interioredges}
  Let $\tilde{S}=\{l_1,\dots, l_k\}$ be a finite collection of lines in $\lines$. Let
  $\X$ be any compact, connected set in $\tree$.
In $\Wh(\X)$, delete the interior of any edge corresponding to one of the lines  $l_i$.
 If the resulting graph $\Wh(\X)\setminus \tilde{S}$ is disconnected then
  $S=q_*(\tilde{S})$ is a cut set.
\end{proposition}

The proof of this proposition is similar to
\fullref{corollary:edgeset}, but it will also be a special case of the
next proposition. Before moving on, though, let us consider an example
that shows that this proposition does not give a necessary condition for
$S$ to be a cut set.

Consider the
pattern $\lines$ generated by the pair of words $b$ and
$ab\bar{a}\bar{b}$ in $F=\left<a,b\right>$. The Whitehead graph (with
loose ends) for this pattern is shown in
  \fullref{fig:a-abAB}.

  \begin{figure}[h!]
    \centering
    \includegraphics[angle=90,height=\figstandardheight]{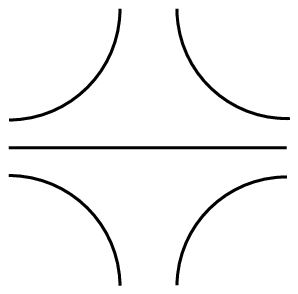}
    \caption{$\Wh(*)\{a,ab\bar{a}\bar{b}\}$ (loose)}
    \label{fig:a-abAB}
  \end{figure}

This graph is connected without cut points, so the decomposition space
is connected. 
We claim that the endpoints of any $b$--line give a cut point in the
decomposition space. For instance, the $b$--line through the identity
vertex has endpoints $b^\infty$, $b^{-\infty}$ in $\bdry \tree$. Let $A=q(b^\infty)=q(b^{-\infty})$.

Let $*$ be
the identity vertex.
Let \[\hat{B}=\cup_{m\in\mathbb{Z}}\shadow^*_\infty(b^ma),\]
that is, $\hat{B}$ consists of all the boundary points $\xi$ of $\tree$ such that the first occurrence of $a$ or $\bar{a}$ in the
geodesic from the identity to $\xi$ is an $a$.

Now $\hat{B}$ is open, and every line of $\lines$ with one
endpoint in $\hat{B}$ has both endpoints in $\hat{B}$.
Let $B=q(\hat{B})$; the preimage is $\hat{B}=q^{-1}(B)$, so $B$ is open in
$\decomp$.
Similarly, let $B'$ be the image in $\decomp$ of the boundary points of $\tree$ such that the first occurrence of $a$ or $\bar{a}$ in the
geodesic from the identity is an $\bar{a}$.

$\decomp=A\cup B\cup B'$, and $A=\closure{B}\setminus
B=\closure{B'}\setminus B'$, so $A$ is a cut point.

For any compact, connected $\X$, $\Wh(\X)$ looks like a
circle with a number of disjoint chords (see
\fullref{fig:whb2-a-abAB}). 

  \begin{figure}[h!]
    \centering
    \includegraphics[angle=90]{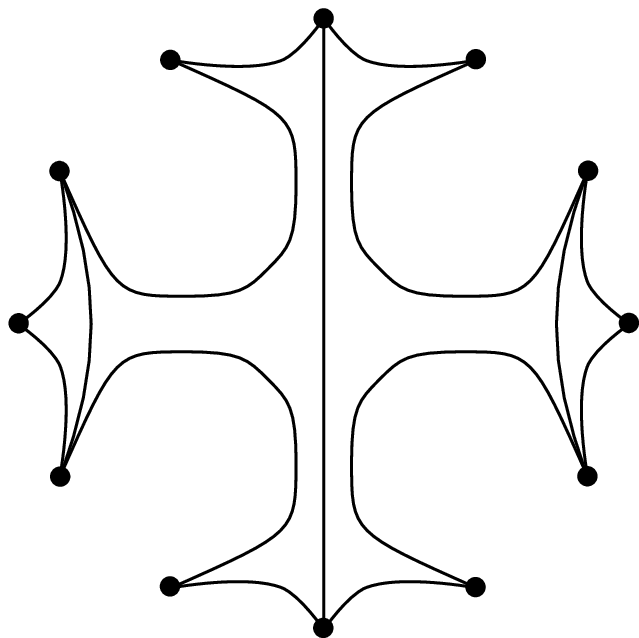}
    \caption{$\Wh(N_1(*))$}
    \label{fig:whb2-a-abAB}
  \end{figure}
The edges of the circle
correspond to $ab\bar{a}\bar{b}$--lines, and the chords correspond to
$b$--lines. This graph has no cut points, so deleting the interior of
an edge does not disconnect it.

This example has shown that to decide if $\tilde{S}$ gives a cut set, it is
not enough to delete the interiors of edges in a Whitehead graph.

There will be several different notions of deleting parts of Whitehead
graphs, so let us standardize notation. Let
$\X\subset\Y\subset\closure{\tree}$, and let $l\in\lines$ and
$\tilde{S}\subset\lines$. Let $e$ be an edge of $\tree$ incident to
exactly one vertex of $\X$.

The edge $e$ corresponds to a vertex in $\Wh(\X)$. The graph
$\Wh(\X)\setminus e$
is obtained from $\Wh(\X)$ by deleting this vertex, but retaining the
incident edges as \emph{loose ends at} $e$.

If $v$ is a vertex of $\tree$ that is distance 1 from $\X$, then there
is a unique edge $e$ with one endpoint equal to $v$ and the other in
$\X$. Define $\Wh(\X)\setminus v=\Wh(\X)\setminus e$.

Similarly, $\Wh(\X)\setminus \Y$ is obtained from $\Wh(\X)$ by deleting each vertex
of $\Wh(\X)$ that corresponds to an edge in $\Y$. 
Visualizing Whitehead graphs in the tree, $\Wh(\X) \setminus \Y$ is the
portion of $\Wh(\Y)$ that passes through the set $\X$. 

$\Wh(\X)\setminus l$ is obtained from $\Wh(\X)$ by deleting the interior
of the edge corresponding to $l$, if such an edge exists.
Similarly, obtain $\Wh(\X)\setminus\tilde{S}$ by deleting the interiors
of any edges corresponding to a line in $\tilde{S}$.

$\Wh(\X)\setminus \ddot{l}$ is obtained from $\Wh(\X)$ by deleting the interior of the edge corresponding to $l$ as well as the two
vertices that are its endpoints, retaining loose ends at these vertices.

Consider the line pattern $\lines$ generated by $ab\bar{a}\bar{b}$ and
$b$. Let $l$ be the $b$--line through the identity vertex $*$. Let
$\X=*$ and let
$\Y=[\bar{b},b]$. Figures~\ref{fig:whbabAB}--\ref{fig:whbabAB-Y} illustrate
our different notions of deleting from $\Wh(\X)$.

\begin{figure}[ht]
\begin{minipage}[b]{0.45\textwidth}
\labellist \small
\pinlabel $a$ [tl] at 164 84
\pinlabel $\bar{a}$ [tr] at 4 84
\pinlabel $b$ [r] at 82 163
\pinlabel $\bar{b}$ [r] at 82 3
\endlabellist
  \centering
  \includegraphics[height=\figstandardheight]{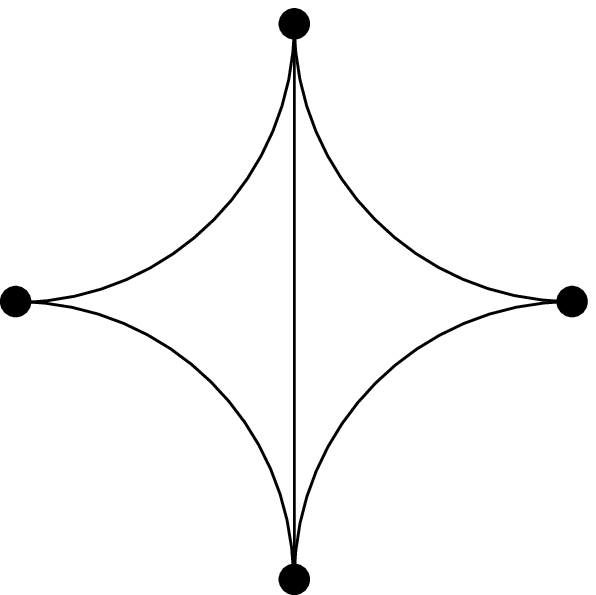}
 \caption{\newline$\Wh(*)$}
\label{fig:whbabAB}
\end{minipage}
\hfill
\begin{minipage}[b]{0.45\textwidth}
\labellist \small
\pinlabel $a$ [tl] at 164 84
\pinlabel $\bar{a}$ [tr] at 4 84
\pinlabel $b$ [r] at 82 163
\pinlabel $\bar{b}$ [r] at 82 3
\endlabellist
  \centering
  \includegraphics[height=\figstandardheight]{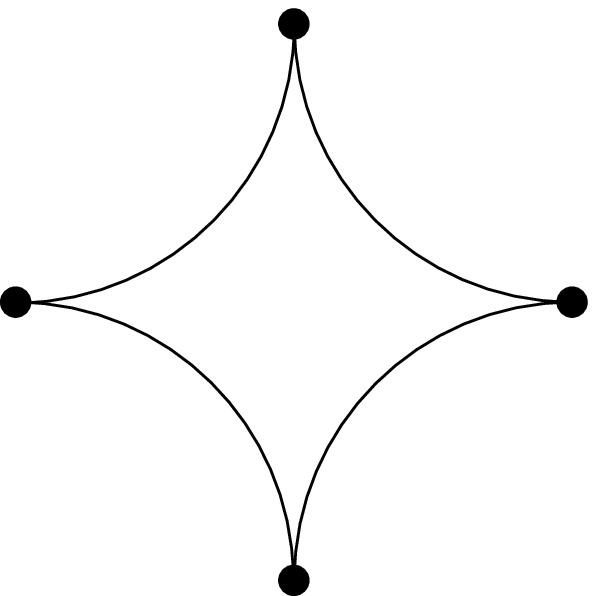}
 \caption{\newline$\Wh(*)\setminus l$}
\end{minipage}
\\
\bigskip
\begin{minipage}[b]{0.45\textwidth}
\labellist \small
\pinlabel $a$ [tl] at 164 80
\pinlabel $\bar{a}$ [tr] at 3 80
\pinlabel $b$ [r] at 79 163
\pinlabel $\bar{b}$ [r] at 78 3
\endlabellist
  \centering
    \includegraphics[height=\figstandardheight]{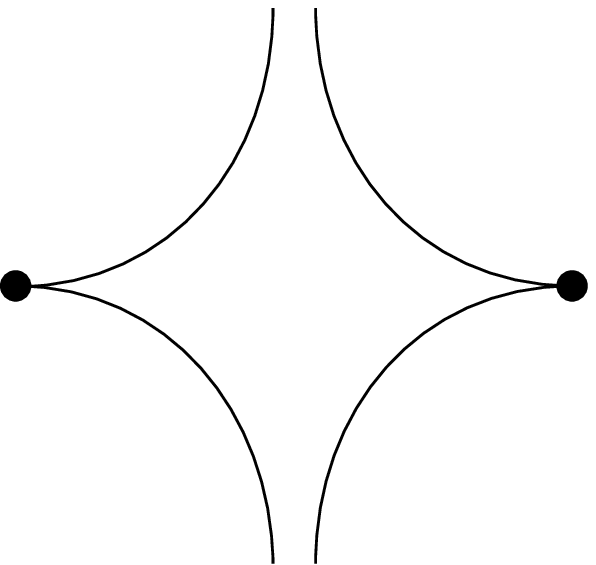}
  \caption{\newline$\Wh(*)\setminus\ddot{l}$}
\end{minipage}
\hfill
\begin{minipage}[b]{0.45\textwidth}
\labellist \small
\pinlabel $a$ [tl] at 164 80
\pinlabel $\bar{a}$ [tr] at 3 80
\pinlabel $b$ [r] at 79 163
\pinlabel $\bar{b}$ [r] at 78 3
\endlabellist
  \centering
    \includegraphics[height=\figstandardheight]{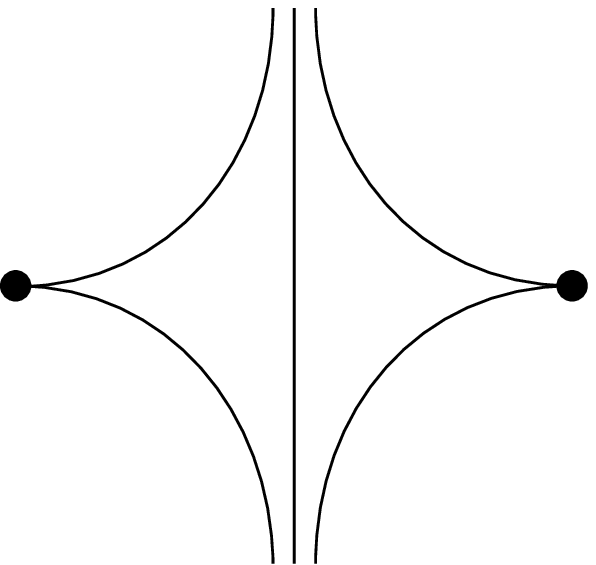}
  \caption{\newline$\Wh(*)\setminus[\bar{b},b]$}
\label{fig:whbabAB-Y}
\end{minipage}
\end{figure}

\begin{lemma}[Hull Determines Connectivity]\label{lemma:connectedhull}
  Let $S$ be a nonempty, finite subset of $\decomp$ that is not just a single bad point. Let $\hull$ be the convex hull of $q^{-1}(S)$. There is a bijection between connected components of $\Wh(\hull)$ and connected components of $\decomp\setminus S$.
\end{lemma}
\begin{proof}
Components $\mathcal{A}_i$ of $\tree\setminus \hull$ are the kind of sets in \fullref{lemma:basicconnectivity}. 
Therefore, $q(\bdry \mathcal{A}_i)$ is connected in $\decomp$.
The set $\bdry \mathcal{A}_i$ is open in $\bdry \tree$. 
For a subcollection $\{\mathcal{A}_{i_j}\}$ corresponding to a
connected component of $\Wh(\hull)$, we have that $q(\bigcup_j \bdry \mathcal{A}_{i_j})$ is an open connected set in $\decomp\setminus S$, as in the proof of \fullref{lemma:basicconnectivity}.
The complement of this set in $\decomp\setminus S$ is either empty or
is a union of sets of a similar form, corresponding to other connected
components of $\hull$. Thus, $q(\bigcup_j \bdry A_{i_j})$ is closed, and is therefore a connected component of $\decomp\setminus S$.
\end{proof}

Pick any vertex $*\in \tree$. If $\xi\in \decomp$ is a bad point, the previous argument applies if we take $\hull$ to be the ray $[*,\xi]$. 
If $\Wh(*)$ is connected without cut points then $\Wh(\hull)$ is connected. Therefore: 

\begin{corollary}
 No bad point of $\decomp$ is a cut point.
\end{corollary}

In a sense, \fullref{lemma:connectedhull} achieves our goal of
relating the topology of the decomposition space to generalizations of
the Whitehead graph.
However, this generalized Whitehead graph is infinite. 
In the next two sections we show that the same information can be
obtained from a finite portion of this Whitehead graph.

\subsection{Identifying Cut Points and Cut Pairs}
The previous corollary tells us that any cut point is a good point, so
its preimage in $\bdry \tree$ is a pair of points. 
We have a similar situation if there is a cut pair consisting of two
bad points; the preimage of such a set in $\bdry\tree$ is a pair of
points.
In both cases, the convex hull is a line.

Suppose $g\in F\setminus\{1\}$ is cyclically reduced with
$\hull^+=g^\infty$ and $\hull^-=g^{-\infty}$.
Let $\hull$ be the convex hull of these two points.
Let $\X=[*,g)$ be the segment joining the identity vertex to the $g$
vertex in $\tree$.

We know, by \fullref{lemma:connectedhull}, that the connected
components of $\decomp\setminus q(\{\hull^-,\hull^+\})$ are in bijection
with components of $\Wh(\hull)$. 
We can construct $\Wh(\hull)$ by splicing together $g$-translates of 
$\Wh(\X)\setminus \hull$.

$\Wh(\X)\setminus\hull$ is $\Wh(\X)\setminus\{e,ge\}$ for some edge $e$
incident to $*$ and
$ge$ incident to $g=g*$, so  $\Wh(\X)\setminus\hull$ has a collection of loose ends at $e$
and at $ge$.
The action of $g$ identifies $\Wh(\X)\setminus\hull$ with
$\Wh(g\X)\setminus\hull$, which has loose ends at $ge$ and $g^2e$.
The line pattern determines for us a splicing map for splicing the loose ends of
$\Wh(\X)\setminus\hull$ at $ge$ to the loose ends of
$\Wh(g\X)\setminus\hull$ at $ge$.

It is an easy consequence of the hypothesis that $\Wh(*)$ is connected
without cut vertices that for any segment $[*,g^k)\subset\hull$, every
component of $\Wh([*,g^k))\setminus \hull$ contains a loose end at $e$
(and a loose end at $g^ke$). 
Thus, the number of components of $\Wh(\hull)$ is bounded above by the
number of components of $\Wh(*)\setminus\hull$.
To bound the number of connected components of $\Wh(\hull)$ below we need to
know if distinct connected components of $\Wh(\X)\setminus
\hull$ become connected when we splice on more translates.

Let $P$ be a partition of the loose ends of
$\Wh(\X)\setminus\hull$ at $e$ that is at least as coarse as
connectivity in $\Wh(\X)\setminus\hull$, ie,  if two loose ends belong to the
same connected component of $\Wh(\X)\setminus\hull$ then they belong in
the same subset of the partition.

Let $|P|$ be the number of subsets in the partition;
$P$ is \emph{nontrivial} if $|P|>1$.

Since $P$ is at least as coarse as connectivity, every vertex and edge in $\Wh(\X)\setminus\hull$ is connected to loose
ends in exactly one subset of $P$.
Let $P'$ be the partition of the loose ends of
$\Wh(\X)\setminus\hull$ at $ge$ such that two loose ends are in the
same subset of the partition if and only if they are connected to
loose ends at $e$ in a common subset of the partition $P$.

The $g$ action determines a partition $gP$ of the loose ends
of $\Wh(g\X)\setminus\hull$ at $ge$ by pushing forward the partition $P$.

We say the partition $P$ is \emph{compatible with the
  splicing map} if there is a bijection between subsets of the
partitions of $P'$ and $gP$ and the splicing map splices edges in a
subset of $P'$ to edges in the corresponding subset of $gP$.

The trivial partition is always compatible with the splicing map, but
this gives us no information. Another obvious partition to consider
would be the partition that comes from connectivity in
$\Wh(\X)\setminus\hull$.
The is the partition in which two loose ends of
$\Wh(\X)\setminus\hull$ at $e$ belong to the same subset
of the partition if and only if they belong to the same connected
component of $\Wh(\X)\setminus\hull$.
Suppose this partition is compatible with the splicing map. 
This would mean that two loose ends of $\Wh(\X)\setminus\hull$ at $ge$
in the same connected component of $\Wh(\X)\setminus\hull$ must splice
to two loose ends of $\Wh(g\X)\setminus\hull$ at $ge$ in the same
connected component of $\Wh(g\X)\setminus\hull$, so splicing
introduces no new connectivity.
In this case it follows that for all $k\geq 1$ the number of connected components of
$\Wh([*,g^k))\setminus\hull$ is equal to the number of connected
components of $\Wh(\X)\setminus\hull$.

However, this is not always the case. 
Splicing may introduce new
connectivity.
Compatibility of the partition controls how much new connectivity is introduced.
If we have a partition compatible with the splicing map, then, after
splicing, the partition $P$ is at least as coarse as connectivity in 
$\Wh([*,g^2))\setminus\hull$. 
Moreover, $P$ will still be compatible with the splicing map at
$g^2e$, so we may continue by induction to show:

\begin{proposition}
Suppose $P$ is a partition that is compatible
with the splicing map.
Then for any
segment $\Y=[*,g^k)$ of $\hull$, the number of
connected components of $\Wh(\Y)\setminus\hull$ is greater than or
equal to $|P|$.
\end{proposition}

Given a compatible partition, there are two cases to consider. 
If for some $\Y$ we have no free edges in
$\Wh(\Y)\setminus\hull$ then the number of connected components of
$\Wh(\hull)$ is greater than or equal to $|P|$. In
particular, if $|P|\geq 2$, then $q(\{\hull^+,\hull^-\})$ is a bad
cut pair in $\decomp$.

Any particular line $l\in\lines$ overlaps with $\hull$ for distance at
most 
\[|g| \times (2+\text{maximum number of consecutive $g$'s in a
generating word for } \lines.)\]
This is infinite if and only if $g$ is a generator of the line pattern.
Thus, $g$ is a generator of the line pattern if and only if
in every $\Y$ there is a free edge in
$\Wh(\Y)\setminus\hull$, since in this case $\hull=\closure{l}$ where
$l\in\lines$ is the $g$--line through the identity.

In this case we could have chosen the partition $P$ so that one of the
subsets is the singleton consisting of the loose end of the free edge.
The partition $P'$ also has a subset that is a singleton, consisting of the
other loose end of the free edge.
Such a partition has a \emph{segregated free edge}.

We do not see the free edge in $\Wh(\hull)$, so in general we can only
conclude that $\Wh(\hull)$ has at least $|P|-1$ connected components.
If $|P|-1\geq 2$ then $q(\{\hull^+,\hull^-\})=q_*(l)$ is a cut point
in $\decomp$.

\begin{proposition}[$q(\{g^\infty,\,g^{-\infty}\})$ Cut Set Criterion]\label{proposition:gcriterion}
  Let $g\in F\setminus\{1\}$ be an element of the free group.
With notation as above, let $P$ be the finest partition that is compatible
  with the splicing map and at least as coarse as connectivity in $\Wh(\X)\setminus\hull$. Then:
  \begin{enumerate}
  \item If $P$ is trivial then $q(\{g^\infty,\,g^{-\infty}\})$ is not a cut
    set.
\item If $P$ is nontrivial and has no segregated free edge then
$q(\{g^\infty,\,g^{-\infty}\})$ is a bad cut pair.
\item If $P$ has a segregated free edge and $|P|=2$ then $q(\{g^\infty,\,g^{-\infty}\})$ is not a cut
    set.
\item If $P$ has a segregated free edge and $|P|>2$ then
 $q(\{g^\infty,\,g^{-\infty}\})$ is a cut point.
 \end{enumerate}
\end{proposition}
\begin{proof}
  If $P$ is trivial then $\Wh([*,g^2))\setminus\hull$ is connected, so $q(\{g^\infty,\,g^{-\infty}\})$ is not a cut
    set.
Similarly, if $|P|=2$ and there is a segregated free edge then
$\hull=\closure{l}$ for $l\in\lines$ and  $\Wh([*,g^2])\setminus l$ is
connected, so $q(\{g^\infty,\,g^{-\infty}\})$ is not a cut
    set.

In the other cases, $\Wh(\hull)$ has multiple components, so $q(\{g^\infty,\,g^{-\infty}\})$ is a cut
    set.
\end{proof}

The proposition tells us that given a $g$ we can decide if
$q(\{g^\infty,\,g^{-\infty}\})$ is a cut set. We call this a
\emph{periodic} cut set.
Next we show that if there are cut points or cut pairs then there are
periodic cut sets:

\begin{proposition}\label{proposition:cutpaircriterion}
If $\decomp$ has cut points or cut pairs then there is some $R$
depending on $\lines$ and some $g$ with $|g|\leq R$ such that
$q(\{g^\infty,\,g^{-\infty}\})$ is a cut set.
\end{proposition}

To identify cut points we just need to apply
\fullref{proposition:gcriterion} to the generators of $\lines$, so in
this case it
is sufficient to take $R$ to be  the length of
the longest generator of $\lines$. 
The work of proving \fullref{proposition:cutpaircriterion} lies in
finding an $R$ that works for the cut pair case:

\begin{lemma}\label{lemma:periodicbadpair}
  If $q(\{\hull^+,\hull^-\})$ is a cut pair then there is some $R$
  depending on $\lines$ and some $g\in F\setminus\{1\}$ with $|g|\leq R$ such that
  $q(\{g^\infty,\,g^{-\infty}\})$ is a cut set.
\end{lemma}

Note that  $q(\{g^\infty,\,g^{-\infty}\})$ is either a cut point or a
bad cut pair.

\begin{proof}
 Let $\hull$ be the convex hull of $\{\hull^+,\hull^-\}$.
We may assume that $\hull$ contains the identity vertex $*$.

Use $\#$ to denote number of connected components.

Every connected component of $\Wh(*)\setminus\hull$ contains an edge,
so the number of components is at most the complexity of $\Wh(*)$.

For any segment $\X$ of $\hull$ we have:
\[2\leq \#\Wh(\hull)\leq\#\left(\Wh(\X)\setminus\hull\right)\leq\#\left(\Wh(*)\setminus\hull\right)\leq \text{complexity
  of }\Wh(*)\]

Number the vertices of $\hull$ consecutively with integers with $*=v_0$
and index increasing in the $\hull^+$ direction. 
Number the edges of
$\hull$ so that $e_i$ is incident to $v_{i-1}$ and $v_i$.
We consider these edges oriented in the direction of increasing index.
An oriented edge of $\tree$ comes with a label that is a generator or
inverse of a generator of $F$.

The function $f(i)=\#(\Wh([\hull^{-\infty},v_i])\setminus\hull)$ is
nonincreasing and, for high enough $i$, stabilizes at $\#\Wh(\hull)$.
Since we started with a cut pair, for high enough $i$ there is no
free edge in $\Wh([\hull^{-\infty},v_i])\setminus\hull$.
After changing by an isometry and relabeling, if necessary, we may
assume that $i=0$ is ``high enough'' in the previous two statements.

Fix a numbering from 1 to $c= \#\Wh(\hull)$ on the components of $\Wh(\hull)$.
At each $v_i$ we get a partition $P_i$ into $c$ subsets of the loose ends of $\Wh(v_i)\setminus\hull$ at
$e_i$ by connectivity in $\Wh(\hull)$.
Similarly, we get a partition $P_i'$ of the  loose ends of $\Wh(v_i)\setminus\hull$ at
$e_{i+1}$.
These partitions are at
least as coarse as connectivity in $\Wh([v_i,v_j])\setminus\hull$
for any $j\geq i$.

By construction, the splicing map at
$e_{i+1}$ connecting loose ends of $\Wh(v_i)\setminus\hull$ at
$e_{i+1}$ to loose ends of $\Wh(v_{i+1})\setminus\hull$ at $e_{i+1}$ is
compatible with the partitions $P_i'$ and $P_{i+1}$.

For each edge pair $(e_i,e_{i+1})$ there is a corresponding label pair
$L_i$ that gives a nontrivial word of length two in $F$. There are
$2n(2n-1)$ such words.

Let $m$ be the number of partitions of (complexity of $\Wh(*)$)
things into $c$ nonempty subsets.
Consider the segment $[v_0,v_R]$, where $R=2n(2n-1)m$. 
Some label pair appears at least $m$ times. 
Let $\{i_j\}_{j=1}^m$ be a set of indices such that the $L_{i_j}$
are the same.

Let $g_{j,k}$ be the element of $F$ that takes $v_{i_j}$ to
$v_{i_k}$. 

If we fix $P_{i_1}$ we get a map of the elements $g_{1,k}$ into the set
of all possible partitions by $g_{1,k}\mapsto g_{1,k}P_{i_1}$, so for some $1\leq j<k$ we have $g_{1,j}$
and $g_{1.k}$ mapping to the same partition.
Therefore, $g_{j,k}P_{i_j}=P_{i_k}$.

$g=g_{i,j}$ is then the desired element.
\end{proof}

\begin{remark}
  In the preceding proof we found an element $g$ such that the $g$--action preserved a partition. 
We did not insist that the $g$--action also fixed the numbering of
components of $\Wh(\hull)$; these may be permuted. 
The proof may easily be modified to fix the numbering, at the expense
of a larger bound on $|g|$.
\end{remark}

\begin{corollary}\label{proposition:pairimpliesbadpair}
  Existence of a cut pair implies existence of a cut point or bad cut pair.
\end{corollary}

\begin{corollary}\label{corollary:longsegments}
  With $R$ as in the previous proposition, for any pair of points
  $\{\hull^+,\hull^-\}\subset\bdry\tree$, if $\X$ is a segment of the
  convex hull $\hull$ of $\{\hull^+,\hull^-\}$ of length greater than
  $R$, and if there are no cut pairs in the decomposition space, then one of the
  following is true:
  \begin{enumerate}
\item $\Wh(\X)\setminus\hull$ is connected.
\item $\Wh(\X)\setminus\hull$ has two components, one of which is a free edge.
  \end{enumerate}
\end{corollary}

\begin{theorem}[Detecting Cut Pairs]\label{theorem:findcutpairs}
There is a finite algorithm for detecting cut pairs in the
decomposition space
\end{theorem}
\begin{proof}
Given a list of words generating a line
pattern, apply Whitehead automorphisms, if necessary, to eliminate cut vertices.
If the graph becomes disconnected, stop; the decomposition space is disconnected by \fullref{corollary:connected}.

If it is possible to disconnect the Whitehead graph by deleting the
interiors of two edges, stop; these two edges correspond to a cut
pair. In particular, this happens if the Whitehead graph has any
valence two vertices.

Use \fullref{proposition:gcriterion} to check if any of the generators
of the line pattern give a cut point in the decomposition space. If
so, stop.

Let $R$ be
the constant from \fullref{lemma:periodicbadpair}.
The idea now is to check segments of length $R$ to see if we can
find a disconnected Whitehead graph. There are a lot of these. We
streamline the process by only checking those long segments for which
every sub-segment gives a disconnected Whitehead graph.

Let $\X_0=\{*\}$. 

We proceed by induction. Suppose $\X_i$ is defined.

Start with $\X_{i+1}=\X_i$.
Consider pairs of points $v$ and $v'$ such that
$d(v,\X_i)=d(v',\X_i)=1$, such that 
$d(v,*)=d(v',*)=i+1$, and such that $*\in[v,v']$.
If $\Wh(\X_i\cap [v,v'])\setminus [v,v']$ is
not connected, add $v$ and $v'$ to $\X_{i+1}$.

Continue until stage $k=1+\lceil \frac{R}{2}\rceil$.
Apply \fullref{corollary:longsegments} and \fullref{proposition:cutpaircriterion}:
there are pairs $v$ and $v'$ in $\X_k\setminus \X_{k-1}$ with
$*\in[v,v']$ such that
$\Wh(\X_{k-1}\cap[v,v'])\setminus [v,v']$ has more than one component
that is not just a free edge if and only if there are cut pairs in the
decomposition space. 

\end{proof}

\begin{corollary}
  If a Whitehead graph for a line pattern has the property
  that deleting any pair of vertices leaves at most one free edge and
  at most one other connected component, then the decomposition space
  has no cut pairs. 
\end{corollary}

Unfortunately, this corollary does not apply if a Whitehead graph has
more than one edge between a pair of vertices. Indeed, consider the
pattern in $F_2=\left<a,b\right>$ generated by the word
$a^2ba^2\bar{b}^2$. The Whitehead graph in \fullref{fig:badcomplete} is reduced and contains the
complete graph on the four vertices, but $q(\{a^\infty,a^{-\infty}\})$
is a cut pair, as is evident from \fullref{fig:badcomplete2}.

\begin{figure}[ht]
\begin{minipage}[b]{0.37\textwidth}

  \centering
 \includegraphics[height=\figstandardheight]{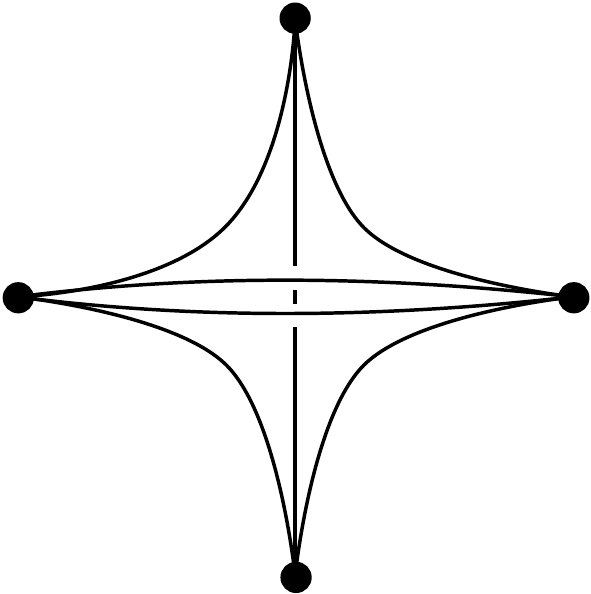}
 \caption{$\Wh(*)\{a^2ba^2\bar{b}^2\}$}
\label{fig:badcomplete}
\end{minipage}
\hfill
\begin{minipage}[b]{0.62\textwidth}

  \centering
  \includegraphics[height=\figstandardheight]{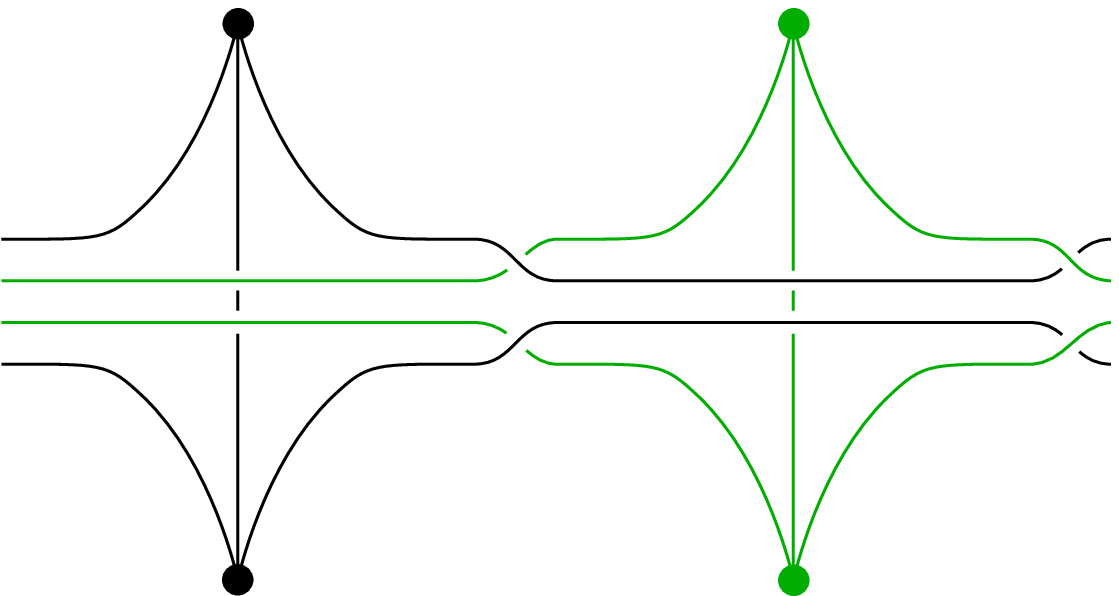}
 \caption{ $\Wh([*,a])\{a^2ba^2\bar{b}^2\}\setminus [a^{-1},a^2]$}
\label{fig:badcomplete2}
\end{minipage}

\end{figure}

\subsection{Cut Sets When There are No Cut Pairs}
Let $S$ be a finite set in the decomposition space, and let $\hull$ be
the convex hull of $q^{-1}(S)$.
\fullref{lemma:connectedhull} tells us that $S$ is a cut set if and
only if $\Wh(\hull)$ is disconnected.
We will pass to a finite subset of $\hull$ whose Whitehead graph
contains the same connectivity information.

Define the \emph{core} $\core$ of $q^{-1}(S)$, to be the smallest closed, connected set such that $\hull\setminus \core$ is a collection of disjoint infinite geodesic rays $\ray_j\from [1,\infty]\to \closure{\tree}$.
We use $\ray_j(0)$ to denote the vertex of the core that is adjacent to
$\ray_j(1)$.

Let $\xi$ be a point in $\bdry\tree$. 
If $q(\xi)$ is either a cut point or a bad point that is a member of a
cut pair, it is not hard to see that there is a geodesic ray $\ray$
with $\ray(\infty)=\xi$ such that
$\Wh(\ray([1,\infty))\setminus\ray(0)$ is not connected.

Conversely, if there exists a geodesic ray
$\ray\from [0,\infty]\to\closure{\tree}$ with $\ray(\infty)=\xi$ such that
$\Wh(\ray([1,\infty))\setminus\ray(0)$ is not connected, then, arguing
as in the proof of \fullref{lemma:periodicbadpair}, $q(\xi)$ is either
a cut point or is a bad point that belongs to a cut pair.

If there are no cut points or cut pairs, then
$\Wh(\ray([1,\infty))\setminus\ray(0)$ is connected for any ray
$\ray$.

\begin{proposition}\label{proposition:nobadpoints}
Suppose $\xi$ is a point in $\bdry{\tree}$ such that $q(\xi)$ is
a bad point that is not a member of  a cut pair. Then $q(\xi)$ is
not a member of any
minimal finite  cut set. In particular, if $\decomp$ has no cut pairs then no
bad point belongs to any minimal finite cut set.
\end{proposition}
\begin{proof}
The assumption that $q(\xi)$ is not a member of a cut pair implies that for any ray $\ray\from [0,\infty]\to \closure{\tree}$
with $\ray(\infty)=\xi$, the Whitehead graph $\Wh(\ray([1,\infty]))\setminus \ray(0)$ is
connected.

Let $S$ be a finite cut set in $\decomp$ with $q(\xi)\in S$. 
Let $\hull$ and $\core$ be the hull and core of $q^{-1}(S)$,
respectively. 
Consider the ray $\ray$ that is the
component of $\hull\setminus \core$ containing $\xi$.

Components of $\decomp\setminus S$ are in bijection with components of
$\Wh(\hull)$, which, in turn, are in bijection with components of
$\Wh(\hull\setminus \ray([1,\infty]))$,
since $\Wh(\ray([1,\infty]))\setminus\ray(0)$ is connected.
This is just the hull of
$q^{-1}(S\setminus\{q(\xi)\})$.

Thus, $S\setminus\{q(\xi)\}$ is still a cut set, so $S$ was not minimal. 
\end{proof}

For a finite collection of lines $\tilde{S}=\{l_1,\dots, l_k\}\subset\lines$, the core is a finite tree. 
The convex hull minus the core is a collection of $2k$ disjoint rays:
\[\{\ray_i^\e\from [1,\infty]\to \tree\mid\lim_{t\to
  \infty}\ray_i^\e(t)=l_i^\e,\,\e\in\{+,-\},\,i=1\dots k\}\] 

\begin{lemma}\label{lemma:openedges}
Let $S$ be a finite set of good points of $\decomp$, none of which is
a cut point. Components of $\decomp\setminus S$ are in bijection with components of
$\Wh(\core)\setminus \tilde{S}$
\end{lemma}
\begin{proof}
Let $\tilde{S}=\{l_1,\dots, l_k\}$.

For each $i$ and $\e$, since $q_*(l_i)$ is not a cut point, $\Wh(\ray_i^\e([1,\infty]))\setminus \ray_i^\e(0)$ is connected.

$\Wh(\hull)$ is obtained from
$\Wh(\core)\setminus \{\ddot{l}_1,\dots, \ddot{l}_k\}$ by splicing on each
$\Wh(\ray_i^\e([1,\infty]))\setminus \ray_i^\e(0)$.

This means to each deleted vertex of $\Wh(\core)\setminus \{\ddot{l}_1,\dots, \ddot{l}_k\}$ we have spliced on a connected graph, so we might have just
as well not deleted those vertices. 
\end{proof}

In fact, we can use the argument of \fullref{lemma:openedges} to
reduce the convex hull even further.
If $\core$ is not just a vertex, then it has some valence one
vertices, that we call \emph{leaves}. The edge connecting a leaf to
the rest of the core is called the \emph{stem}.

Suppose that for some leaf $v$ of $\core$ every line of the $\tilde{S}$
that goes through $v$ goes through the stem of $v$. From $\Wh(v)$,
delete the interiors of edges corresponding to the $l_i$ and the
vertex corresponding to the stem. 
The resulting graph is connected, so connected components of
$\Wh(\core)\setminus\tilde{S}$ are in bijection with connected
components of $\Wh(\core\setminus\{v\})\setminus\tilde{S}$.

Thus, we may \emph{prune} some leaves off of the core without changing
the connectivity of the Whitehead graphs. 

If $\tilde{S}$ is not an edge cut set then we may prune the core down to a
well defined nonempty tree $\pcore$, the \emph{pruned core}, such that every leaf
contains a line of $\tilde{S}$ that does not go through the stem.

If $\tilde{S}$ is an edge cut set then the core can be pruned down to a pair
of adjacent vertices, both of which look like prunable leaves.
In this case define $\pcore$ to be these two vertices.

\begin{proposition}\label{proposition:edgeset}
 An edge cut set that does not contain a cut point is minimal.
\end{proposition}
\begin{proof}
  Let $\tilde{S}=\{l_1,\dots, l_k\}$ be the set of lines of
  $\lines$ going through an edge $e$ of $\tree$, so that $S=q_*(\tilde{S})$ is
  an edge cut set.
Let $\pcore$ be the pruned core. 
There are two connected components of $\Wh(\pcore)\setminus\tilde{S}$;
they lie on opposite sides of $e$. 
By \fullref{lemma:openedges} these correspond to two connected
components of $\decomp\setminus S$.

Each of the $l_i$ has one endpoint in
each component, so if any $l_i$ is omitted from
the set the two components will have a point in common.
\end{proof}

\begin{corollary}\label{corollary:topologicalmorality}
  If $\decomp$ has no cut pairs, the good points and bad
  points are topologically distinguished.
\end{corollary}
\begin{proof}
  Bad points are the points that do not belong to any minimal finite
  cut set. Good points are the points that do.
\end{proof}

\begin{proposition}\label{proposition:twocomponents}
 Let $S$ be a minimal finite cut set, none of whose elements are members of a cut pair.
There are  exactly two connected components of $\decomp\setminus S$.
\end{proposition}
\begin{proof}
By \fullref{proposition:nobadpoints}, $S$ consists of good points.
Let $\tilde{S}=\{l_1,\dots, l_k\}=q_*^{-1}(S)$.
Components of $\decomp\setminus S$ are in bijection with components of
$\Wh(\core)\setminus \tilde{S}$. 
This is a finite graph, so $\decomp\setminus S$ has only finitely many components.

Let $A_1,\dots, A_m$ be a list of the components of $\decomp\setminus S$.

If $q_*(l_i)$ is not a limit point of $A_j$ in $\decomp$ then $A_j$ is still a connected component in $\decomp\setminus (S\setminus q_*(l_i))$.
This contradicts minimality of $S$, so each point of $S$ is a limit point in $\decomp$ of every $A_j$.
 This implies that for each $i$ and $j$, at least one of the points $l_i^+$ and $l_i^-$ is a limit point of $q^{-1}(A_j)$.

Now $\hull\setminus \core$ is a collection of disjoint rays $\ray_i^\e$.
The graph $\Wh(\ray_i^\e([1,\infty]))\setminus \ray_i^\e(0)$ is connected, so  no $l_i^+$ or $l_i^-$ is a limit point of more than one $q^{-1}(A_j)$.

Thus, there are exactly two components $A_1$ and $A_2$ of $\decomp\setminus S$, and each line $l_i$ has one endpoint in $q^{-1}(A_1)$ and the other in $q^{-1}(A_2)$. 
\end{proof}

\begin{corollary}\label{corollary:bothcomponents}
  Let $S$ be a minimal finite cut set that is not an edge cut set,
  none of whose elements are members of a cut pair.
For
  every vertex $v\in \pcore$, the portion of
  $\Wh(\pcore)\setminus\tilde{S}$ at $v$ contains an edge from each component of
    $\Wh(\pcore)\setminus\tilde{S}$.
\end{corollary}
\begin{proof}
If $v$ is a leaf such that the portion of
$\Wh(\pcore)\setminus\tilde{S}$ at $v$ belongs to a single component of
$\Wh(\pcore)\setminus\tilde{S}$ then $v$ should have been pruned
off.

If $v$ is not a leaf, $\pcore\setminus\{v\}$ has at least two components. If the Whitehead
graph over one of those components sees only one component of
$\Wh(\pcore)\setminus\tilde{S}$ then it would have been
possible to prune it off. Thus, every component of
$\pcore\setminus\{v\}$ must see two components of
$\Wh(\pcore)\setminus\tilde{S}$.
There are only two components of
$\Wh(\pcore)\setminus\tilde{S}$, so both must be able to
connect to all components of $\pcore\setminus\{v\}$. In particular,
they must connect through $v$.
\end{proof}

\subsection{Indecomposable Cut Sets}
In this section we will assume that the decomposition space has no cut
pairs.

Our ultimate goal is to construct a cubing quasi-isometric to
a bounded valence tree.
For this purpose, we will need to choose a collection of cut sets in a
such a way that there is a bound on the number of cut sets in the
collection that cross any fixed cut set in the collection.

Cut sets with disjoint pruned cores do not cross, so we could control
crossings if we could control the diameters of the pruned cores of the
cut sets in some collection.

The following example shows that cut sets of a fixed size can have
pruned cores with arbitrarily large diameter. 
We subsequently introduce
the property of \emph{indecomposability} to rule out this kind of bad behavior.

Let $\lines$ be the line pattern in $F=\left<a,b\right>$ generated by
the words $ab\bar{a}\bar{b}$, $a$ and $b$.
The edge cut sets have size three. It can be shown that these are the
only cut sets of size three and there are none smaller, see \fullref{example:complete}.
It is also possible to find minimal cut sets of size four.
Pick any two of the edge cut sets that share a line. The four lines of
the symmetric difference are a minimal cut set.
\fullref{fig:indecomposable} shows the line pattern. The two dashed
lines indicate edge cut sets of size three. The four thickened lines
make up the cut set of size four that is the symmetric difference.
 There is no bound on
the size of the pruned core of such a cut set, nor on the number of
such cut sets that cross each other. 

\begin{figure}[h]
  \centering
  \includegraphics[width=\textwidth]{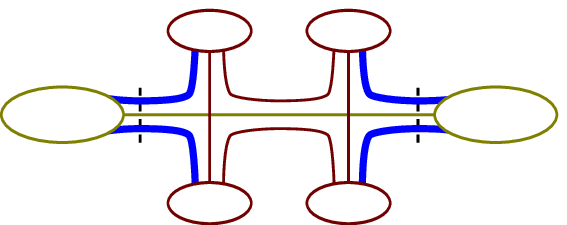}
  \caption{A problematic minimal cut set}
  \label{fig:indecomposable}
\end{figure}

We say that a minimal finite cut set $S\subset\decomp$ is \emph{decomposable} if there are 
minimal cut sets $Q$ and $R$ such that:
\begin{enumerate}
\item $Q$ and $R$ are non-crossing,
\item $|Q|<|S|$ and $|R|<|S|$,
\item $S=Q\Delta R=(Q\setminus R)\cup (R\setminus Q)$
\end{enumerate}

A minimal finite cut set $S$ is \emph{indecomposable} if it is not decomposable.

The smallest cut sets in $\decomp$ are indecomposable since there are no smaller cut sets to decompose them into.

\begin{lemma}\label{proposition:decomposable}
  Suppose $S$ is a finite minimal cut set and the pruned core $\pcore$ of $\tilde{S}$ has an interior vertex $v$ such that $\Wh(v)\setminus \pcore$ has exactly two components, one of which is a free edge, and no lines of $\tilde{S}$ go through $v$. Then $S$ is decomposable.
\end{lemma}
\begin{proof}

  \begin{figure}[h]
\labellist \small
\pinlabel $l$ [t] at 46 17
\pinlabel $v$ [br] at 56 22
\pinlabel $Q$ [b] at 5 18
\pinlabel $R$ [b] at 110 18
\pinlabel $\X_1$ at 16 33
\pinlabel $\X_0$ at 16 7
\pinlabel $\Y_0$ at 99 7
\pinlabel $\Y_1$ at 99 33
\endlabellist
    \centering
    \includegraphics[width=.7\textwidth]{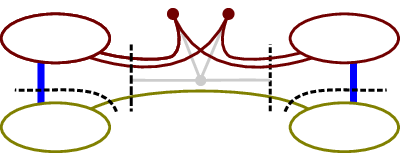}
    \caption{Schematic diagram of decomposable cut set}
    \label{fig:decomposable}
  \end{figure}

Let $l$ be the line of $\lines$ corresponding to the free edge in $\Wh(v)\setminus \pcore$. 
Let $\tilde{Q}$ be $l$ and the lines of $\tilde{S}$ on one side of $\pcore\setminus\{v\}$, and let $\tilde{R}$ be $l$ and the lines of $\tilde{S}$ on the other side of $\pcore\setminus\{v\}$.

Then $Q\cap R=q_*(l)$, and $S=Q\Delta R$.

Let $A^0$ and $A^1$ be the components of $\decomp\setminus S$. 
The line $l$ does not belong to $\tilde{S}$, so we may assume that
$q_*(l)\in A^0$.
Let $\X$ and $\Y$ be the two components of $\pcore\setminus\{v\}$. 
We may assume $Q$ is on the $\X$ side.

Let $\X_0$ be the portion $\X$ corresponding to $A^0$, and define
$\X_1$, $\Y_0$ and $\Y_1$ analogously, see \fullref{fig:decomposable}. 
The edge of $\Wh(\pcore)\setminus\tilde{S}$ corresponding to $l$ is the only connection between $\X_0$ and $\Y_0$, so $\tilde{Q}$ separates $\X_0$ from $\X_1\cup \Y_0\cup \Y_1$.

Thus, $Q$ is a cut set. Moreover, $Q$ is a minimal cut set since every edge corresponding to a
line in $\tilde{Q}$ has one end in $\X_0$ and one end in $\X_1\cup
\Y_0\cup \Y_1$. 

By a similar argument, $R$ is a minimal cut set.

$Q$ and $R$ are non-crossing because the only line of $R$ that has an
endpoint in $\X_0$ is $l=\tilde{Q}\cap\tilde{R}$.

Finally, as there are no cut pairs, we have:
\[3\leq |Q|=|Q\setminus R|+|Q\cap R|=|Q\setminus R|+1\implies |Q\setminus R|\geq2\] 
Thus:
\[|S|=|Q\setminus R|+|R\setminus Q|>1+|R\setminus Q|=|R\cap Q|+|R\setminus Q|=|R|\]
Similarly, $|Q|<|S|$.

\end{proof}

\begin{theorem}[Edge cut sets are indecomposable]\label{theorem:indecomposableedges}
Suppose we have chosen a free basis of $F$ such that $\Wh(*)$ is minimal complexity. Then edge cut sets are indecomposable.
\end{theorem}
\begin{proof}
Let $e$ be an edge of $\tree$. Let
$\tilde{S}$ be the lines of $\lines$ that cross $e$. Let $S=q_*(\tilde{S})$.

$S$ is minimal by \fullref{proposition:edgeset}.
Suppose $S$ decomposes into $Q$ and $R$. 
We must have $Q\cap R\neq \emptyset$, otherwise $Q$ and $R$ are proper subsets of
$S$ that are cut sets, contradicting minimality of $S$.
Since $Q$ and $R$ do not cross and $S\setminus R=Q\setminus R$, $S$ does
not cross $R$. 
Thus, since they are minimal, $R$ does not cross $S$. 
Therefore, $R\setminus S=R\cap Q$ is contained in one component of
$\decomp\setminus S$. 
This means that $q_*^{-1}(Q\cap R)=\tilde{Q}\cap\tilde{R}$ is contained in one component of $\tree\setminus e$. 

Let $*$ be the vertex of $\tree$ incident to $e$ on the $\tilde{Q}\cap \tilde{R}$ side.

It is possible that pruning the cores of $\tilde{Q}$ or $\tilde{R}$ would remove $*$. Let the
\emph{partially pruned core} of $\tilde{Q}$, $\ppcore_{\tilde{Q}}$, be the result of pruning the
core of $\tilde{Q}$ as much as possible without pruning off $*$. 
Note $\ppcore_{\tilde{Q}}=\ppcore_{\tilde{R}}$, so we may just call it $\ppcore$.

\[|R\setminus Q|+|Q\cap R|=|R| < |S|=|Q\setminus R|+|R\setminus Q|\]

So $|Q\cap R|<|Q\setminus R|$. Similarly, $|Q\cap R|<|R\setminus Q|$.

There are two connected components of $\Wh(\ppcore)\setminus\tilde{Q}$,
call them component 0 and component 1. 
Since $Q$ and $S$ do not cross, everything on the side of $e$ opposite $\tilde{Q}\cap \tilde{R}$
belongs to a single component.

First suppose $\ppcore=*$.
Suppose the edge $e$ oriented away from $*$ has label
$x\in\mathcal{B}\cup\bar{\mathcal{B}}$; suppose the corresponding
vertex in $\Wh(\ppcore)\setminus\tilde{Q}$ is in component 1.
Suppose the vertex
corresponding to the edge labeled $\bar{x}$ is in component 0.
Then the Whitehead automorphism that pushes the vertices in $\Wh(*)$ in
component 1 through $x$ changes the valence at $x$ from
$|S|=|Q\setminus R|+|R\setminus Q|$ to
$|Q\cap R|+|R\setminus Q|$.
Since $|Q\cap R|<|Q\setminus R|$ this contradicts the assumption that
the Whitehead graph had minimal complexity.

Conversely, if the vertex $\bar{x}$ is in component 1 push
\[Z=\{x\}\cup\{\text{vertices of component } 0\}\] through $x$ and get a contradiction.

Now suppose $\ppcore$ is not just $*$. Then there is some leaf $v\neq *$.
Suppose the stem of $v$ (oriented away from the leaf) has label $x\in\mathcal{B}\cup\bar{\mathcal{B}}$, and suppose the 
vertex in $\Wh(\ppcore)\setminus\tilde{Q}$ corresponding to $\bar{x}$ is in component 1.

\fullref{fig:ppcore} shows a schematic diagram of $\Wh(\ppcore)$.

The labeling in the diagram is as follows:
\begin{itemize}
\item $\X_0$ = the portion of component 0 on the $v$ side of
 the stem.
\item $\X_1$ = the portion of component 1 on the $v$ side of
 the stem.
\item $\Y_0$ = the portion of component 0 between the stem of
 $v$ and $e$.
\item $\Y_1$ = the portion of component 1 between the stem of
 $v$ and $e$.
\item $\mathcal{Z}$ = everything on the side of $e$ opposite $\tilde{Q}\cap\tilde{R}$.
\item lowercase letters represent the number of lines with endpoints
  in the specified regions with:
  \begin{itemize}
  \item $a$, $b$, $c$ and $h$ counting the lines of
    $\tilde{Q}\cap\tilde{R}$
\item $d$ and $i$ counting the lines of $\tilde{Q}\setminus\tilde{R}$
\item $e$ and $j$ counting the lines of $\tilde{R}\setminus\tilde{Q}$
\item $f$ and $g$ counting the lines not in $\tilde{Q}\cup\tilde{R}$
  crossing the stem.
  \end{itemize}
\end{itemize}

\begin{figure}[h]
\labellist \small
\pinlabel $v$ [r] at 33 36
\pinlabel $a$ [l] at 13 36
\pinlabel $b$ [tr] at 28 53
\pinlabel $c$ [br] at 28 19
\pinlabel $d$ [tr] at 39 55
\pinlabel $e$ [br] at 39 16
\pinlabel $f$ [bl] at 48 60
\pinlabel $g$ [tl] at 48 12
\pinlabel $h$ [r] at 114 33.5
\pinlabel $i$ [b] at 183 57
\pinlabel $j$ [t] at 183 16
\pinlabel $\X_0$ at 22 64
\pinlabel $\X_1$ at 22 9
\pinlabel $\Y_0$ at 115 64
\pinlabel $\Y_1$ at 115 9
\pinlabel $\mathcal{Z}$ at 205 37
\pinlabel $Q$ [b] at 170 65
\pinlabel $R$ [t] at 170 9
\pinlabel $S$ [b] at 176 66
\endlabellist
  \centering
  \includegraphics[width=\textwidth]{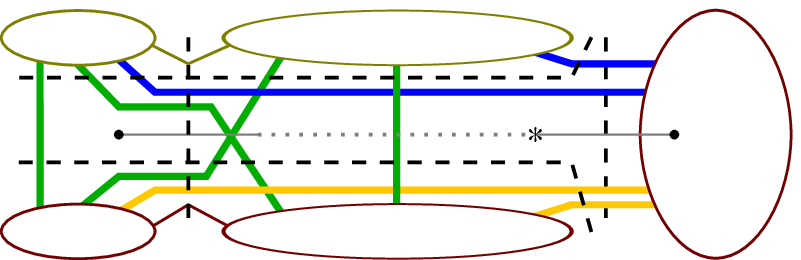}
  \caption{Schematic diagram for $\Wh(\ppcore)$}
\label{fig:ppcore}
\end{figure}

Since $v\neq *$ is a leaf of $\ppcore$ we must have $a>0$ and
$\X_0$ and $\X_1$ nonempty.

Minimality of $Q$ implies that $\Wh(\ppcore)\setminus\tilde{Q}$ has
exactly two connected components. In the diagram they are
$\X_0\cup\Y_0$ and
$\X_1\cup\Y_1\cup\mathcal{Z}$.

$\X_0\cup\Y_0$ must belong to a connected component,
so $\Y_0=\emptyset$ if and only if $f=0$.
$\Y_0=\emptyset$ also implies $c=h=i=0$. Now $d+i=|Q\setminus
R|$, so this would imply $d>0$.

$R$ is also a minimal cut set, so $\Wh(\ppcore)\setminus\tilde{R}$ must
have exactly two components. In the diagram they are
$\X_0\cup\Y_0\cup\mathcal{Z}$ and
$\X_1\cup\Y_1$.
Since $\X_1\cup\Y_1$ is connected,
$\Y_1=\emptyset$ if and only if $g=0$.

Thus, we have:
\begin{enumerate}
\item $\Y_0=\emptyset \iff f=0 \implies d>0,\,c=h=i=0$
\item $\Y_1=\emptyset \iff g=0 \implies e>0,\,b=h=j=0$
\end{enumerate}

The Whitehead automorphism that pushes
\[Z=\{x\}\cup\{\text{vertices of $\Wh(v)$ in
component 0 of $\Wh(\ppcore)\setminus\tilde{Q}$}\}\]
 through the stem changes the valence of
vertex $x$ from $b+c+d+e+f+g$ to $a+c+e+g$. 
By our minimal complexity assumption, we must therefore have $a\geq
b+d+f$.

Now $|Q\setminus R|>|Q\cap R|\geq a+b+c\geq 2b+c+d+f$, which means
that $|Q\setminus R|>d$, so $i>0$.

We will now change $Q$ and $R$ to a new decomposing pair $Q'$ and $R'$ for $S$ with
strictly smaller partially pruned core.

Let $\tilde{Q}'\setminus \tilde{R}'$ be the lines of
$\tilde{Q}\setminus \tilde{R}$ that do not pass through $v$.

Let $\tilde{R}'\setminus\tilde{Q}'$ be the rest of $\tilde{S}$.

Let $\tilde{Q}'\cap\tilde{R}'$ be $\tilde{Q}\cap\tilde{R}$ minus the
lines contributing to $a$ and $b$ plus the lines contributing to $f$.

\begin{figure}[h]
\labellist \small
\pinlabel $v$ [r] at 33 36
\pinlabel $a$ [l] at 13 36
\pinlabel $b$ [tr] at 28 53
\pinlabel $c$ [br] at 28 19
\pinlabel $d$ [tr] at 39 55
\pinlabel $e$ [br] at 39 16
\pinlabel $f$ [bl] at 48 60
\pinlabel $g$ [tl] at 48 12
\pinlabel $h$ [r] at 114 33.5
\pinlabel $i$ [b] at 183 57
\pinlabel $j$ [t] at 183 16
\pinlabel $\X_0$ at 22 64
\pinlabel $\X_1$ at 22 9
\pinlabel $\Y_0$ at 115 64
\pinlabel $\Y_1$ at 115 9
\pinlabel $\mathcal{Z}$ at 205 37
\pinlabel $Q'$ [b] at 170 65
\pinlabel $R'$ [t] at 170 9
\pinlabel $S$ [b] at 176 66
\endlabellist
  \centering
  \includegraphics[width=\textwidth]{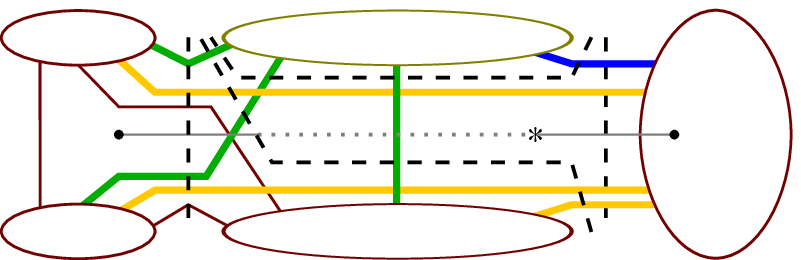}
  \caption{Schematic diagram for modified $\Wh(\ppcore)$}
\label{fig:ppcore2}
\end{figure}

\begin{align*}
  |S|-|R'|&=|{Q}'\setminus{R}'|-|{Q}'\cap{R}'|\\
&=|Q\setminus R|-d-(|Q\cap R|-a-b+f)\\
&=|Q\setminus R|-|Q\cap R|+a+b-d-f\\
&\geq |Q\setminus R|-|Q\cap R| +(b+d+f)+b-d-f\\
&=|S|-|R|+2b\geq |S|-|R|>0
\end{align*}

A similar computation shows $|S|-|Q'|\geq |S|-|Q|+2(b+d)>0$.

We must show that $Q'$ and $R'$ are non-crossing minimal cut sets.

The components of $\Wh(\ppcore)\setminus\tilde{Q}'$ are $\Y_0$
and $\X_0\cup\X_1\cup\Y_1\cup\mathcal{Z}$. To
see that the latter is connected, note that $a>0$, $e+j>0$ and either
$\Y_1=\emptyset$ or $g>0$.

Thus, $Q'$ is a minimal cut set since $\Wh(\ppcore)\setminus\tilde{Q}'$
has exactly two connected components and every line of $\tilde{Q}'$
goes from one component to the other.

By similar considerations, $R'$ is a minimal cut set since
$\Wh(\ppcore)\setminus\tilde{R}'$ has components
$\Y_0\cup\mathcal{Z}$ and $\X_0\cup\X_1\cup\Y_1$.

That $Q'$ and $R'$ are non-crossing follows from the observation:
\[\Y_0\cap\left(\X_0\cup\X_1\cup\Y_1\right)=\emptyset\]

We have not added anything to $\tilde{Q}'\cap\tilde{R}'$ except
possibly some lines going through $v$, so the new partially pruned core is contained in the old one minus the
vertex $v$.

If $\bar{x}$ is in component 0, repeat the argument with the roles of
$Q$ and $R$ reversed and reach a similar conclusion.

Thus, by repeating this process, we can reduce the partially pruned
core until we find some decomposing pair $Q$ and $R$
so that the partially pruned core is just $*$. We have already
seen that that leads to a contradiction, so $S$ is indecomposable.
\end{proof}

\begin{theorem}[Indecomposables are bounded]\label{theorem:boundeddecomposable}
  The diameter of the pruned core of an indecomposable cut set $S$ is
  bounded in terms of $\mathcal{L}$ and $|S|$.
\end{theorem}
\begin{proof}
If $\pcore$ is a point or two points we are done. Otherwise it is a tree with leaves. Each leaf contains a line from $\tilde{S}$ that does not go through its stem, so there are at most $|S|$ leaves. 

Suppose $\X$ is a segment of $\pcore$ that does not have any lines of $\tilde{S}$ going through it.
By \fullref{corollary:bothcomponents}, at every vertex of $\pcore$ there are edges of $\Wh(\pcore)\setminus\tilde{S}$ from both components.
Since $S$ is indecomposable, by \fullref{proposition:decomposable} it is not the case that one of these components is a free edge.
Now apply \fullref{corollary:longsegments} and conclude that there is a bound $R$ on the length $\X$. 

Similarly, if $\X$ is a segment of $\pcore$ that meets exactly one of the $l_i$ then it has length bounded by $R$.

It follows that the diameter of $\pcore$ is at most $2R(|S|-1)$.
\end{proof}

\section{Rigidity}\label{sec:rigidity}
\subsection{The Problem with Cut Pairs}\label{sec:flexibility}
If $\decomp$ has cut pairs then it has either a cut point or a bad cut
pair, by \fullref{proposition:pairimpliesbadpair}. In either case,
there is a cut set such that the preimage in $\bdry \tree$ is two
points $\{g^\infty, g^{-\infty}\}$. The convex hull $\hull$ of these two points is a line, and
$\Wh(\hull)$ has multiple components, $A_1,\dots, A_k$.
For each $i$, let $\X_i$ be the union of components of $\tree\setminus
\hull$ corresponding to $A_i$.

The action of $g$ may permute these components, but $g^{k!}$ fixes
them.

Let $\phi\from\tree\to\tree$ be the quasi-isometry:
\[
\phi(x)=\begin{cases}
  g^{k!}x & \text{if } x\in\mathcal{X}_1\\
x &\text{otherwise}
\end{cases}
\]

This ``shearing'' quasi-isometry moves the $\mathcal{X}_1$ component
along $\hull$, fixing the rest of the tree. 

It is not hard to see that $\phi^n$ is not bounded distance from an
isometry for $n\neq 0$. Since $F$ acts by isometries it follows that
$F\phi^a$ and $F\phi^b$ are not the same coset of $F$ in
$\QIgp(F,\lines)$ when $a\neq b$, so $F$ is infinite index subgroup. 

It is possible to show directly that $\phi$ could not be conjugate
into an isometry group. Alternatively, notice that we can stack
shearing quasi-isometries to produce a sequence of quasi-isometries
with unbounded multiplicative quasi-isometry constants, see \fullref{fig:shear}.
Take an element $h$ of $F$ such that $h\hull$ is contained in the
$\mathcal{X}_1$ component with $h\mathcal{X}_1\subset \mathcal{X}_1$.

The desired sequence of quasi-isometries is $(\Phi_i)$, where:
\[\Phi_i=\phi^i\circ (h\phi^i\bar{h})\circ (h^2\phi^i\bar{h}^2)\circ \cdots\]
That is, for any $x\in \tree$ there exists some $j$ such that $x\in
h^j\X_1\setminus h^{j+1}\X_1$.
For $m>j$ $h^m\phi^i\bar{h}^m(x)=x$, so:
\[\Phi_i(x)=\phi^i\circ (h\phi^i\bar{h})\circ
(h^2\phi^i\bar{h}^2)\circ \cdots\circ (h^j\phi^i\bar{h}^j)(x)\]

\begin{figure}[h]
\labellist 
\small
\pinlabel $h^3\hull$ [r] at 0 137
\pinlabel $h^2\hull$ [r] at 0 97
\pinlabel $h\hull$ [r] at 0 57
\pinlabel $\hull$ [r] at 0 17
\pinlabel $\stackrel{\Phi_3}{\longrightarrow}$ at 160 80
\endlabellist
  \centering
  \includegraphics[height=\figstandardheight]{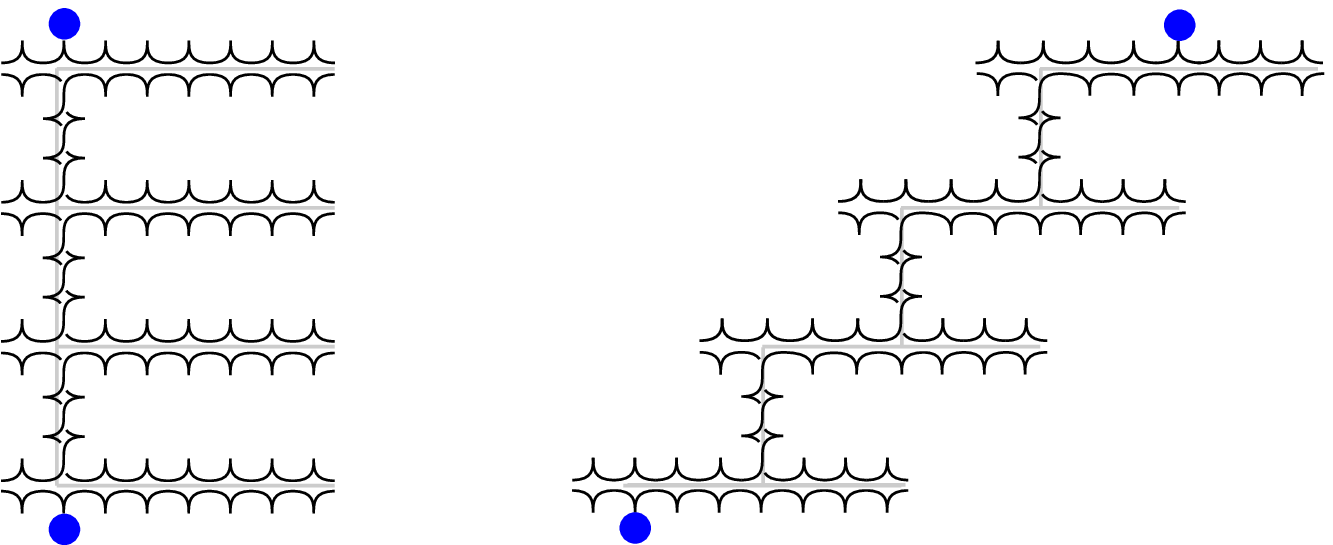}
  \caption{Shearing}
  \label{fig:shear}
\end{figure}

\subsection{Rigidity When There are No Cut Pairs}
Let $\lines$ be a line pattern in $F$.
Choose a free basis $\mathcal{B}$ for $F$ so that $\Wh(*)=\Wh_\mathcal{B}(*)\{\lines\}$ has minimal complexity, and let $\tree$ be the Cayley graph of $F$ with respect to $\mathcal{B}$.
Assume that $\decomp=\decomp_\lines$ has no cut pairs.
We will construct a cubing $X$, a quasi-isometry $\phi\from \tree\to
X$, and an isometric action of $\QIgp(F,\lines)$ on $X$ that agrees
with $\phi\QIgp(F,\lines)\phi^{-1}\subset\Isomgp(X)$, up to bounded
distance , completing the proof of the \hyperlink{main}{Main Theorem}.
The action of $F$ on $X$ will be cocompact, implying that
$\QIgp(F,\lines)$ has a complex of groups structure.

\subsubsection{Constructing the Cubing}
Let $b$ be the maximum valence of a vertex in $\Wh(*)$. 
Let $\{S_i\}_{i\in I}$ be the collection of indecomposable cut sets of size at most $b$ in $\decomp$.
For each $i\in I$, $S_i$ is a finite minimal cut set, by definition, and, since there are no cut pairs, $\decomp\setminus S_i$ has exactly two connected components, by \fullref{proposition:twocomponents}.

Let the two connected components of $\decomp\setminus S_i$ be $A_i^0$ and $A_i^1$. 
Let \[\Sigma=\{A^0_i\}_{i\in I}\cup\{A^1_i\}_{i\in I}\]

Recall from \fullref{sec:cubings} that from this information we define a graph as follows:

 A vertex is a subset $V$ of $\Sigma$ such that:
\begin{enumerate}
\item For all $i\in I$ exactly one of $A^0_i$ or $A^1_i$ is in $V$.
\item If $C\in V$ and $C'\in\Sigma$ with $C\subset C'$ then $C'\in V$. 
\end{enumerate}

Two vertices are connected by an edge if they differ by only one set in $\Sigma$. 

This gives a graph; it remains to select a path connected component of this graph to be the 1-skeleton of the cubing.

Define a \emph{bad triple} $\bar{x}=\{x_1,x_2,x_3\}$ to be an
unordered triple of distinct bad points in $\decomp$. 

There are no bad points in minimal cut sets, so for any bad triple and
any $S_i$, $\bar{x}\subset A_i^0\cup A_i^1$.
We let $\bar{x}$ decide democratically whether it will associate with
$A_i^0$ or $A_i^1$: say $\bar{x}\in A_i^\e$ if at least two of the
$x_j$'s are in $A_i^\e$.

Define $V_{\bar{x}}=\{A_i^\e\in\Sigma\mid \bar{x}\in A_i^\e\}$. This is
a vertex of $X$. 
Define the 0-skeleton of the cubing to be the set $X^{(0)}$ of all vertices that are
connected by a finite edge path to $V_{\bar{x}}$ for some bad triple $\bar{x}$.

The following lemma replaces Lemma~3.4 of \cite{Sag95}.
\begin{lemma}\label{lemma:finiteseparation}
 For any bad triples $\bar{x}$ and $\bar{y}$, there are only finitely many $S_i$ separating them.
\end{lemma}
\begin{proof}
Let $\bar{x}=\{x_1,x_2,x_3\}$ and $\bar{y}=\{y_1,y_2,y_3\}$ be bad
triples.

The preimage $q^{-1}(\bar{x})=\{q^{-1}(x_i)\}_{i=1,2,3}$ consists of three distinct points in $\bdry \tree$.
The convex hull of three points in the boundary of a tree is a
tripod. The core, as previously defined, is the unique vertex that is
the branch point of the tripod. Denote this point $\core_{\bar{x}}$.

It is not hard to see that a cut set $S_i$ separates $\bar{x}$ from
$\bar{y}$ only if the pruned core $\pcore$ of $S_i$ intersects the
finite geodesic segment joining $\core_{\bar{x}}$ and
$\core_{\bar{y}}$ in $\tree$. 

By \fullref{theorem:boundeddecomposable}, there is a uniform bound $a$ on the diameter of the pruned core of any $S_i$. 
Since $\lines$ is locally finite, this means there is a uniform bound $c$ on the number of $S_i$ such that $*\in\pcore_{S_i}$.
If $\mathcal{Y}$ is any finite collection of vertices in $\tree$, the number of $S_i$ such that $\pcore_{S_i}\cap\mathcal{Y}\neq\emptyset$ is at most $c|\mathcal{Y}|$.

Thus, the number of $S_i$ separating $\bar{x}$ from $\bar{y}$ is at
most $c\cdot \left(1+d_\tree(\core_{\bar{x}},\core_{\bar{y}})\right)$.
\end{proof}

Add edges to the 0-skeleton as above to get the 1-skeleton $X^{(1)}$
of the cubing. With \fullref{lemma:finiteseparation} replacing
Lemma~3.4 of \cite{Sag95}, the following theorem follows by the same proof as in \cite{Sag95}:
\begin{theorem}{\rm{\cite[Theorem~3.3]{Sag95}}}
  $X^{(1)}$ is connected.
\end{theorem}

The rest of the construction of the cubing follows as in \fullref{sec:cubings}.

\begin{remark}
  We are forced to use this alternate way of choosing the vertices of
  the cubing because every
  good point in $\decomp$ belongs to infinitely many of the cut sets. 

Also, \fullref{lemma:finiteseparation} is false if one tries to use
just bad points instead of bad triples. Two bad points are separated
by infinitely many of the $S_i$.
\end{remark}

\begin{remark}
  For a fixed vertex $v\in\tree$, there are uncountably many bad
  triples $\bar{x}$ with $\core_{\bar{x}}=v$. However, these give only
  finitely many distinct vertices $V_{\bar{x}}$ in $X$, because the
  $V_{\bar{x}}$ can only differ in the finitely many coordinates $i$ such that
  the pruned core of $S_i$ contains $v$. Even this is an over
  count. If $S_e$ is an edge cut set associated to an edge $e$
  incident to $v$, then every bad triple with $\core_{\bar{x}}=v$ lies
  in the same component of $\decomp\setminus S_e$. If our set of
  indecomposables is exactly the collection of edge cut sets then the
  cubing is isomorphic to the tree $\tree$. 
\end{remark}

Notice $X$ is defined in terms of the topology of $\decomp$, so we have:
\begin{lemma}
 Any homeomorphism of $\decomp$ induces an isomorphism of $X$.
\end{lemma}


\subsubsection{Estimates on the Cubing}
Recall from the proof of \fullref{lemma:finiteseparation}, we have a
bound $a$ on the diameter of the pruned core of any $S_i$, and there
is a $c$ such that
if $\mathcal{Y}$ is any finite collection of vertices in $\tree$, the number of $S_i$ such that $\pcore_{S_i}\cap\mathcal{Y}\neq\emptyset$ is at most $c|\mathcal{Y}|$.

$S_i$ and $S_j$ are non-crossing if their pruned cores are disjoint, so we have a uniform bound $c(2n)^{\frac{a}{2}}$ on the number of $S_j$ that cross a fixed $S_i$.

A $k$-cube in $X$ corresponds to a collection of $k$ pairwise crossing
cut sets, so the cubing is finite-dimensional.

Pick a vertex $x\in X$. Let $e$ and $e'$ be edges incident to
$x$. There are distinct  hyperplanes $H_e$ and $H_{e'}$ associated to
these edges. Since $e$ and $e'$ are incident to a common vertex, there
is no third hyperplane separating $H_e$ from $H_{e'}$. Therefore, the
valence of a vertex in $X$ is bounded by the maximum size of a
subcollection $\{S_i\}_{i\in J}$ of the indecomposable cut sets such that
for any $j$ and $k$ in $J$, there is no $i\in I$ such that $S_i$
separates $S_j$ and $S_k$. If $S_j$ and $S_k$ have disjoint pruned
cores then there is an edge cut set separating them, so the maximum
size of the set $J$ is at most $c(2n)^{\frac{a}{2}}$. Thus, $X$ is
uniformly locally finite.

A hyperplane $H$ corresponds to an equivalence class of edges in $X$. 
The 1-neighborhood of $H$ is the set of vertices that are endpoints of
these edges. If $k$ is the number of hyperplanes crossing $H$, then
the 1-neighborhood of $H$ has at most $2^{k+1}$ vertices and diameter
at most $k+1$. Crossing hyperplanes correspond to crossing cut sets,
so $k$ is at most $c(2n)^{\frac{a}{2}}$.

\subsubsection{The Rigidity Theorem}
\begin{theorem}
  $X$ is quasi-isometric to $\tree$.
\end{theorem}
\begin{proof}
For each edge $e\in\tree$ there is a corresponding edge cut set
$S_e$. By construction, $S_e\in \{S_i\}$, so in the cubing $X$ there is a corresponding hyperplane
$H_e$. Define $\phi(e)$ to be the set of vertices in the
1-neighborhood of $H_e$. Recall from the previous section that this is
a set of boundedly many vertices with bounded diameter.

$d_X(\phi(e),\phi(e'))$ is the number of hyperplanes separating $H_e$ and $H_{e'}$. This is at least the number of edges separating $e$ from $e'$ in $\tree$, which is $d_\tree(e,e')$, and at most the number of $\{S_i\}$ such that $\pcore_{S_i}$ meets the geodesic between $e$ and $e'$ in $\tree$, which is bounded by $c\cdot d_\tree(e,e')$.
This shows that $\phi$ is a quasi-isometric embedding. 

Suppose there is a vertex $x\in X$ not in the image of $\phi$.
This $x$ has some incident edge, corresponding to some $S\in\{S_i\}$.
The hypothesis that $x$ is not in the image of $\phi$ implies that $S$ does not cross any edge cut set, which means that $\pcore_s$ is a single vertex $v\in\tree$.
There are boundedly many such $S$, and the distance from $x$ to $\phi(\tree)$ is less than this bound, so $\phi$ is coarsely onto, hence a quasi-isometry.
\end{proof}

The quasi-isometry $\phi$ gives a collection of quasi-lines $\phi(\mathcal{L})$ in $X$. 
In fact, we can see this collection of quasi-lines directly from $\decomp$. Each good point in $\decomp$ belongs to infinitely many indecomposable cut sets. 
For $l\in\lines$, the collection $\{S \mid S\text{ indecomposable},\, |S|\leq b,\,q_*(l)\in S\}$ corresponds to a collection of hyperplanes in $X$. 
The union of these hyperplanes is coarsely equivalent to $\phi(l)$.

\begin{theorem}[Rigidity Theorem]\label{theorem:rigidity}
  For $i=1,2$, let $F_i$ be a free group with line pattern $\lines_i$.
Let $\decomp_i$ be the decomposition space corresponding to $\lines_i$ in $F_i$.

Suppose, for each $i$, $\decomp_i$ has no cut pairs.

Let $\phi_i\from F_i\to X_i$ be the quasi-isometry to the cube complex constructed above.
Then: \[\phi_2 \QIgp\{(F_1,\lines_1)\to(F_2,\lines_2)\}\phi_1^{-1}=\Isomgp\{(X_1,\phi_1(\lines_1))\to(X_2,\phi_2(\lines_2))\}\]
\end{theorem}
\begin{proof}
  Elements of $\QIgp\{(F_1,\lines_1)\to(F_2,\lines_2)\}$ give homeomorphisms $\decomp_1\to\decomp_2$.
Homeomorphism take indecomposable cut sets to indecomposable cut sets
of the same size and preserve crossing and intersection.
Therefore, we get isometries $X_1\to X_2$ respecting the line patterns. 
\end{proof}

The Rigidity Theorem answers Questions \ref{question:qiequivalence} and \ref{question:qigp} for rigid patterns.

The free group acts on itself by pattern preserving isometries via
left multiplication. Let $*$ be the identity vertex in $\tree$. For
any indecomposable cut set $S$, there is an element $g\in F$ such that
$*\in g(\pcore_S)$. There are only finitely many indecomposable cut
sets of bounded size with $*\in\pcore$, so $F$ acts cocompactly on
$X$. Therefore, $\QIgp(F,\lines)\cong\Isomgp(X,\phi(\lines))$ acts
cocompactly on $X$. 
This gives an explicit finite presentation for $\Isomgp(X,\phi(\lines))$ as a complex of groups.
Moreover, as the $F$ action is already cocompact, we have:

\begin{corollary}
If $\lines$ is a rigid line pattern and if $\QIgp(F,\lines)$ acts on
$X$ with finite stabilizers then $F$ is a finite
index subgroup of $\QIgp(F,\lines)$. 
\end{corollary}

\section{Examples}
\subsection{Whitehead Graph is the Circle}\label{example:circle}
We will show in this section that when the Whitehead graph is a circle we get a quasi-isometrically flexible line pattern.
\begin{theorem}\label{theorem:circle}
For a line pattern $\lines$ in $F$, the following are equivalent:
  \begin{enumerate}
\item Every Whitehead graph $\Wh_\mathcal{B}(*)$ that has no cut
 vertex is a circle. \label{item:allWcircle}
\item Some Whitehead graph $\Wh_\mathcal{B}(*)$ is a circle. \label{item:someWcircle}
\item $\decomp$ is a circle.\label{item:Dcircle}
\item Every minimal cut set of $\decomp$ has size two.\label{item:cutpairsonly}
 \end{enumerate}
\end{theorem}
\begin{proof}
Clearly  (\ref{item:allWcircle})$\implies$(\ref{item:someWcircle}), because Whitehead automorphisms will eliminate cut vertices.

If some Whitehead graph $\Wh_\mathcal{B}(*)$ is a circle then we
can realize the free group $F_n$ as the fundamental group of a surface
with boundary, and the generators of the line pattern $\lines$ as
the boundary labels. We can give this surface a hyperbolic metric so
that the universal cover is just $\tree$ fattened, and the boundary
components are horocycles that are in bijection with the lines of
$\lines$. This gives us a homeomorphism between the decomposition
space and $S^1=\partial \mathbb{H}^2$. Thus
(\ref{item:someWcircle})$\implies$(\ref{item:Dcircle}).

(\ref{item:Dcircle})$\implies$(\ref{item:cutpairsonly}) is a
topological property of circles.

Now, suppose every minimal cut set of $\decomp$ has size two. 
Then $\decomp$ is connected with no cut points.
Choose a free basis $\mathcal{B}$ so
that $\Wh_\mathcal{B}(*)$ is connected without cut points. 
The edges incident to a vertex of $\Wh_\mathcal{B}(*)$ correspond to an edge cut set.
This is a minimal cut set by \fullref{proposition:edgeset}, so by
hypothesis has size two.
Therefore, $\Wh_\mathcal{B}(*)$ is a finite, connected graph with all valences equal
to two, hence, a circle. Thus, (\ref{item:cutpairsonly})$\implies$(\ref{item:allWcircle}).
\end{proof}

\begin{remark}
  Otal proves \cite[Theorem 2]{Ota92} that the decomposition space is
  a circle if and only if the the collection of words can be
  represented as the boundary curves a compact surface. The proof is
  essentially the same.
\end{remark}

\begin{theorem}\label{theorem:allcircles}
Let $F$ and $F'$ be free groups, possibly of different rank. Let
$\lines$ and $\lines'$ be line patterns in $F$ and $F'$, respectively.
Suppose $\decomp_\lines$ is a circle. There is a pattern preserving
quasi-isometry from $F$ to $F'$ if and only if $\decomp_{\lines'}$ is
also a circle.
\end{theorem}
\begin{proof}
The ``only if'' direction is clear, as a pattern preserving
quasi-isometry induces a homeomorphism of decomposition spaces.

Suppose both $\decomp_\lines$ and $\decomp_{\lines'}$ are
circles. 
By \fullref{theorem:circle}, there exist free bases $\mathcal{B}$ of $F$ and
$\mathcal{B}'$ of $F'$ such that $\Wh_\mathcal{B}(*)\{\lines\}$ and
$\Wh_{\mathcal{B}'}(*)\{\lines'\}$ are circles.
  
As in
the proof of \fullref{theorem:circle} we can associate each pattern
with the boundary curves of the universal cover of a surface with
boundary. It is a theorem of Behrstock and Neumann \cite{BehNeu06}
that there are many boundary preserving quasi-isometries of such surfaces.
\end{proof}

For example, recall the \hyperlink{introexample}{example from the Introduction}.
Let $F=F_2=\left<a,b\right>$.

Let $\lines_1$ be the line pattern generated by the word $ab\bar{a}\bar{b}$. 

Let $\lines_2$ be the line pattern generated by the words $ab$ and
$a\bar{b}$. 

For each of these $\Wh(*)\{\lines_i\}$ is a circle, so the two patterns
are quasi-isometrically equivalent.

This example also shows that neither the number of generators of a
line pattern nor the widths of the generators are quasi-isometry invariants.

\subsection{Whitehead Graph is the Complete Graph}\label{example:complete}
Let $K_{2n}$ be the complete graph on $2n$ vertices, the graph consisting of $2n$ vertices with exactly one edge joining each pair of vertices.

Suppose $\lines$ is a line pattern in $F=F_n$ so that for some free basis
$\mathcal{B}$, $\Wh_\mathcal{B}(*)\{\lines\}=K_{2n}$.

The decomposition space $\decomp$ has no cut pairs. 

Suppose $S$ is a minimal finite cut set of $\decomp$ that is
not an edge cut set.
$\Wh(\pcore)\setminus\tilde{S}$ has two components.
The portion of $\Wh(\pcore)\setminus\tilde{S}$ at a leaf contains vertices from both
components.

The portion of $\Wh(\pcore)\setminus\tilde{S}$ at a leaf is a graph
obtained from $K_{2n}$ be deleting a vertex, corresponding to the stem
of the leaf, and interiors of some number of edges coming from lines of
$\tilde{S}$ that go through the leaf but not through the stem. 
The result is a disconnected graph with at least one vertex in each of
the components. 
Thus, we have  partition of $2n-1$ vertices into two subsets, and we
must delete all the edges between them.
The subsets have sizes $m$ and $2n-1-m$, for some $1\leq m\leq 2n-2$,
and the number of edges between them is $m(2n-1-m)\geq 2n-2$.
There are at least two leaves, so $|S|\geq 4n-4 >2n-1$. 
The edge cut sets have size $2n-1$, so our construction of a cubing
uses only the edge cut sets. Thus, the cubing is just the tree $\tree$.

In this case it is easy to compute: \[\QIgp(F,\lines)\cong
\Sym(2n)*_{\Sym(2n-1)}\left(\Sym(2n-1)\times \Sym(2)\right)\] 

Here, $\Sym(2n)$ is the symmetric group on $2n$ objects, stabilizing a
vertex of the tree and permuting the incident edges, and
$\Sym(2n-1)\times \Sym(2)$ is the stabilizer of an edge of $\tree$.

This discussion proves the following theorem:
\begin{theorem}
  Suppose $\lines$ is a line pattern in $F=F_n$ such that
  $\Wh_\mathcal{B}(*)\{\lines\}=K_{2n}$.
Suppose that $F'=F_m$ is another free group, possibly of different rank,
with line pattern $\lines'$. 

There is a pattern-preserving quasi-isometry  $F\to F'$ if and only if $\decomp_{\lines'}$ has the following properties:
  \begin{enumerate}
  \item There are no cut sets of size less than $2n-1$.
\item The collection of cut sets of size $2n-1$ yields a cubing that is a $2n$-valent tree.
\item The induced line pattern in the cubing restricts to the complete graph $K_{2n}$ in the star of any vertex.
  \end{enumerate}
\end{theorem}

For example, the line pattern $\lines$ in $F=F_2$ with basis $\mathcal{B}=\{a,b\}$
generated by $a$, $b$, and $ab\bar{a}\bar{b}$ has Whitehead graph $\Wh_\mathcal{B}(*)\{\lines\}=K_4$.

Compare this to the line pattern $\lines'$ in $F'=F_3$ with basis
$\mathcal{B}'=\{x,y,z\}$ generated by $y$, $zx$, $z\bar{x}\bar{y}$ and
$xy\bar{z}$.
The Whitehead graph $\Wh_{\mathcal{B}'}(*)\{\lines'\}$ looks like two copies of $K_4$
spliced together, see \fullref{fig:twinK4}. 

\begin{figure}[h]
\labellist \small
\pinlabel $x$ [b] at 137 160
\pinlabel $z$ [b] at 50 160
\pinlabel $\bar{x}$ [t] at 50 2
\pinlabel $\bar{z}$ [t] at 137 2
\pinlabel $y$ [r] at 3 79
\pinlabel $\bar{y}$ [l] at 183 79 
\endlabellist
  \centering
  \includegraphics[height=\figstandardheight]{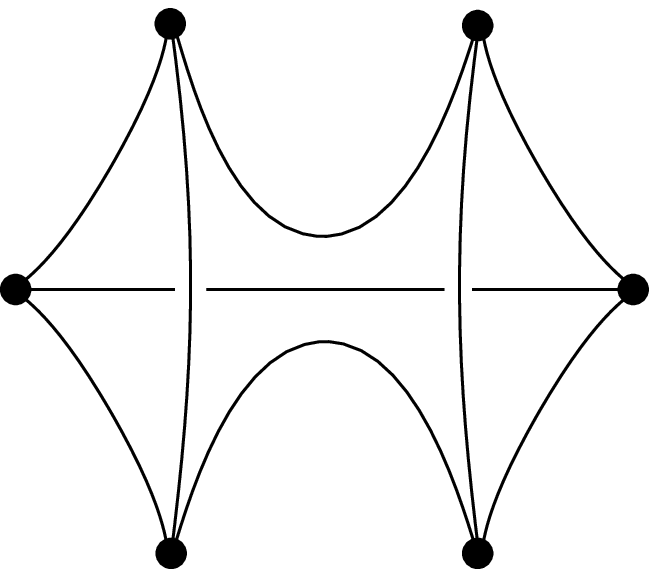}
  \caption{$\Wh_{\{x,y,z\}}(*)\{y, zx, z\bar{x}\bar{y}, xy\bar{z}\}$}
  \label{fig:twinK4}
\end{figure}

It is not hard to show that the smallest cut sets
are the obvious ones of size three. These yield a cubing that is a
4-valent tree, essentially blowing up each vertex of $F_3$ into a pair
of vertices. 

This pattern is quasi-isometric to the $K_4$ pattern in $F_2$.

\subsection{A Rigid Example for which the Free Group is not Finite Index in the
  Group of Pattern Preserving Quasi-isometries}

Consider the line pattern in $F=\left<a,b\right>$ generated by the
words $a$, $b$, and $aba\bar{b}\bar{a}b\bar{a}\bar{b}$.
 Let $\tree$ be
the Cayley graph of $F$ with respect to $\{a,b\}$.

It is easy to check that Whitehead graph in \fullref{fig:notfg} is
reduced and the decomposition space has no cut pairs, so the pattern is
rigid.

\begin{figure}[h]
  \centering
  \includegraphics[height=\figstandardheight]{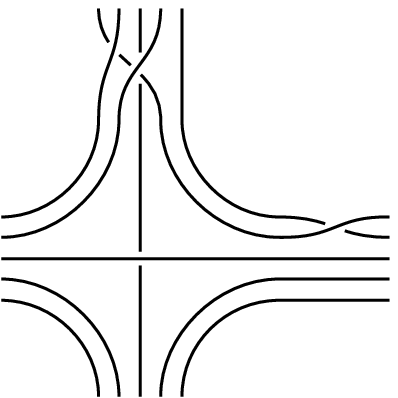}
  \caption{$\Wh(*)\{a,b,aba\bar{b}\bar{a}b\bar{a}\bar{b}\}$ (loose)}
  \label{fig:notfg}
\end{figure}

The edge cut sets have size five. Deleting any vertex of the Whitehead
graph leaves a graph that requires at least three more edges to be
deleted to disconnect the graph. Thus, any other cut sets have size at
least six. As the edge cut sets are the only cut sets of size less
than or equal to five, the rigid cube complex is just the tree $\tree$.

We will show that $F$ is not a finite index subgroup of
$\QIgp(F,\lines)$. Not only are the vertex stabilizers in
$\QIgp(F,\lines)$ not finite, they are not even finitely generated.

Define an isometry $\phi$ of $\tree$ piecewise as follows.
First, note that the automorphism $\alpha$ of $F$ that exchanges $a$ with
$\bar{a}$  preserves the pattern. It inverts $a$, fixes $b$, and takes
$aba\bar{b}\bar{a}b\bar{a}\bar{b}$ to a cyclic permutation of itself.
To the branch of the tree consisting
of words beginning with $b$, apply the automorphism $\alpha$.
To each branch of the tree beginning with $a^nb$
for some $n$, apply  the automorphsim $a^n\circ \alpha \circ
\bar{a}^n$.
Fix the rest of the tree.

The isometry $\phi$ is built piecewise from pattern
preserving automorphisms of $F$.  
It fixes the ``bottom
half'' of $\tree$,
fixes the $b$--line through $a^n$ for each $n$, and reflects each
branch beginning with $a^nb$ through the $b$--line through $a^n$.

There are lines of the pattern that pass through multiple pieces of
the domain of $\phi$, so we check that the $\phi$ is defined
consistently for these lines.
As illustrated in \fullref{fig:notfgmultiple}, the only lines shared
by the bottom half of the tree and the vertical branches are the fixed
$b$--lines (green lines in the figure are fixed).
The reflections in
adjacent vertical branches agree on the two lines they share (the two
thickened blue lines are exchanged).
Therefore, $\phi$ pieces together to give a pattern preserving
isometry. 

\begin{figure}[h]
\labellist
\small
\pinlabel $*$ [tl] at 152 40
\pinlabel $\bar{a}$ [tl] at 40 40
\pinlabel $b$ [tl] at 152 152
\pinlabel $\bar{a}b$ [tl] at 40 152
\endlabellist
  \centering
  \includegraphics[height=2\figstandardheight]{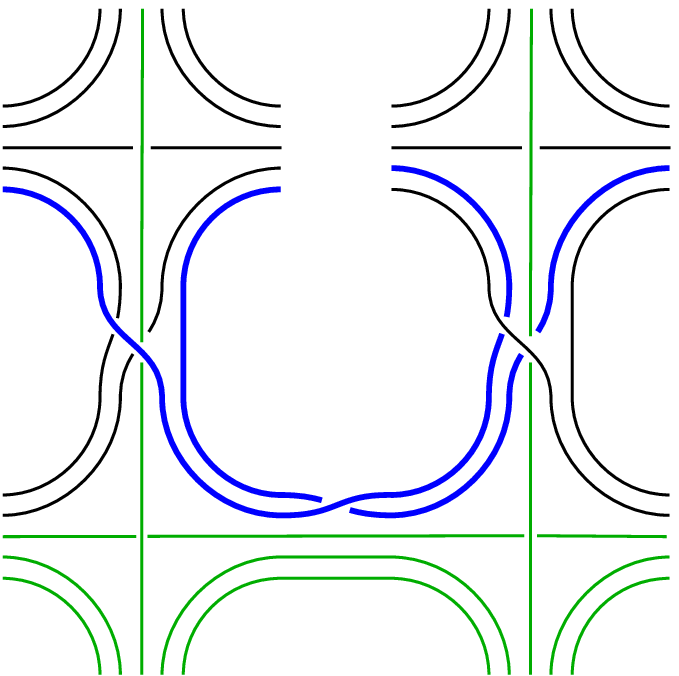}
  \caption{$\Wh([\bar{a}b,b])\{a,b,aba\bar{b}\bar{a}b\bar{a}\bar{b}\}$ (loose)}
  \label{fig:notfgmultiple}
\end{figure}

Thus, for any $n$, $b^n\circ \phi \circ \bar{b}^n$ is a pattern preserving
isometry that fixes every line in the
$n$-neighborhood of the identity vertex, but is not the identity map. 
It follows that the stabilizer of the identity vertex is not finitely generated.

\subsection{A Cube Complex That is Not a Tree}
Finally, we give an example of a rigid line pattern for which our
argument does not produce a cube complex that is a tree.

Consider the line pattern in $F_3=\left<a,b,c\right>$ generated by the
four words $\bar{a}bc$, $\bar{a}cb$, $\bar{a}b^3$ and
$\bar{a}c^3$. 
The Whitehead graph (with loose ends), is shown in
\fullref{fig:bigwh}.

The reader may verify that this is a minimal Whitehead graph and there
are no cut points or cut pairs in the decomposition space. 
In fact, the smallest cut sets are the edge cut sets of size four
corresponding to the $a$--edges. These are the only cut sets of size
four.

The other edge cut sets have size five, so we construct a cube
complex using indecomposable cut sets of size four and five.
\fullref{fig:bigwhcc} depicts the Whitehead graph (along with portions
of the Whitehead graph over two neighboring vertices) with the
1--skeleton of the cube
complex overlaid.

\begin{figure}[h]
\labellist
\small
\pinlabel $a$ [b] at 40 164
\pinlabel $b$ [r] at 0 106
\pinlabel $c$ [l] at 81 106
\pinlabel $\bar{a}$ [t] at 40 0
\pinlabel $\bar{b}$ [l] at 81 41
\pinlabel $\bar{c}$ [r] at 0 41
\endlabellist
 \centering
 \includegraphics{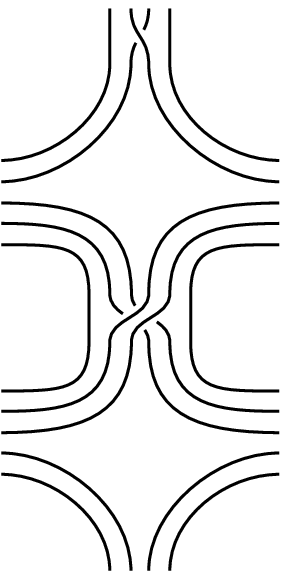}
 \caption{$\Wh(*)\{\bar{a}bc, \bar{a}cb, \bar{a}b^3, \bar{a}c^3\}$}
 \label{fig:bigwh}
\end{figure}

\begin{figure}[h]
 \centering
 \includegraphics{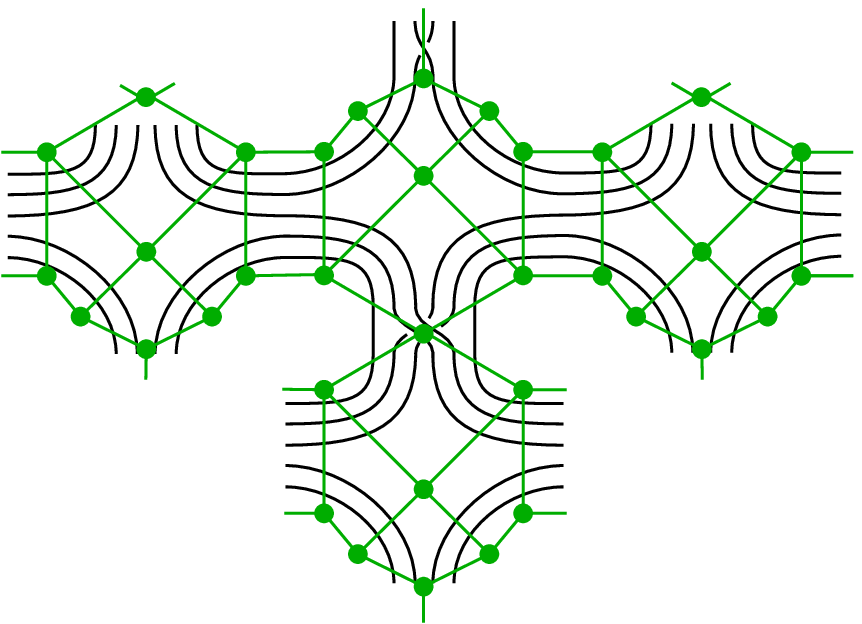}
 \caption{Whitehead graph with cube complex}
 \label{fig:bigwhcc}
\end{figure}

Note that every cut set of size five is crossed by another cut set of
size five. However, the edge cut sets are still topologically
distinguished! They are the cut sets of size five that are
crossed minimally (once) by another cut set of size five. The other
cut sets of size five are crossed by either two or five other cut sets
of size five. 

Had we said, ``build the cube complex associated to the cut sets of size four and
those of size five that are crossed by exactly one other cut set of size
five'' we would have recovered the tree as the cube complex. 

In every example we know, it is possible, after computing the cube
complex, to pick out a topologically distinguished collection of cut
sets whose associated cube complex is a tree. 
We do not know whether this is true in general.


\bibliographystyle{hyperamsplain}
\bibliography{masterbib}

\providecommand{\bysame}{\leavevmode\hbox to3em{\hrulefill}\thinspace}
\providecommand{\MR}{\relax\ifhmode\unskip\space\fi MR }
\providecommand{\MRhref}[2]{%
  \href{http://www.ams.org/mathscinet-getitem?mr=#1}{#2}
}
\providecommand{\href}[2]{#2}
\begin{thebibliography}{10}

\bibitem{BehNeu06}
Jason~A. Behrstock and Walter~D. Neumann, \emph{Quasi-isometric classification
  of graph manifold groups}, Duke Math. J. \textbf{141} (2008), 217--240,
  \href{http://dx.doi.org/10.1215/S0012-7094-08-14121-3}{\nolinkurl{doi:10.121%
5/S0012-7094-08-14121-3}},
  \href{http://arXiv.org/abs/math.GT/0604042}{\nolinkurl{arXiv:math.GT/0604042%
}}.

\bibitem{BriHae99}
Martin~R. Bridson and Andr{\'{e}} Haefliger, \emph{Metric spaces of
  non-positive curvature}, Grundlehren der mathematischen Wissenschaften, vol.
  319, Springer, Berlin, 1999.

\bibitem{CasMac09b}
Christopher~H. Cashen and Nata{\v{s}}a Macura, \emph{Quasi-isometries of
  mapping tori of linearly growing free group automorphisms}, in preparation,
  2010.

\bibitem{Dun85}
Martin~J Dunwoody, \emph{The accessibility of finetly presented groups},
  Invent. Math. \textbf{81} (1985), no.~3, 449--457,
  \href{http://dx.doi.org/10.1007/BF01388581}{\nolinkurl{doi:10.1007/BF0138858%
1}}.

\bibitem{Gro87}
M.~Gromov, \emph{Hyperbolic groups}, Essays in group theory, Math. Sci. Res.
  Inst. Publ., vol.~8, Springer, New York, 1987, pp.~75--263.
  \href{http://www.ams.org/mathscinet-getitem?mr=919829 }{\nolinkurl{MR919829
  (89e:20070)}}

\bibitem{Kap07a}
Michael Kapovich, \emph{Energy of harmonic functions and {G}romov's proof of
  {S}tallings' theorem}, preprint,
  \href{http://arXiv.org/abs/0707.4231}{\nolinkurl{arXiv:0707.4231}}, 2007.

\bibitem{LimTho08}
Seonhee Lim and Anne Thomas, \emph{Covering theory for complexes of groups}, J.
  Pure Appl. Algebra \textbf{212} (2008), no.~7, 1632--1663,
  \href{http://dx.doi.org/10.1016/j.jpaa.2007.10.012}{\nolinkurl{doi:10.1016/j%
.jpaa.2007.10.012}},
  \href{http://arXiv.org/abs/math/0605303}{\nolinkurl{arXiv:math/0605303}}.

\bibitem{LynSch77}
Roger~C. Lyndon and Paul~E. Schupp, \emph{Combinatorial group theory}, Classics
  in Mathematics, Springer-Verlag, Berlin, 2001, Reprint of the 1977 edition.
  \href{http://www.ams.org/mathscinet-getitem?mr=1812024}{\nolinkurl{MR1812024
  (2001i:20064)}}

\bibitem{Man09}
Jason~Fox Manning, \emph{Virtually geometric words and {W}hitehead's
  algorithm}, preprint,
  \href{http://arXiv.org/abs/0904.0724}{\nolinkurl{arXiv:0904.0724}}, 2009.

\bibitem{Mar95}
Reiner Martin, \emph{Non-uniquely ergodic foliations of thin type, measured
  curents and automorphisms of free groups}, Ph.D. thesis, UCLA, 1995.

\bibitem{Nib04}
Graham~A. Niblo, \emph{A geometric proof of {S}tallings' theorem on groups with
  more than one end}, Geom. Dedicata \textbf{105} (2004), 61--76,
  \href{http://dx.doi.org/10.1023/B:GEOM.0000024780.73453.e4}{\nolinkurl{doi:1%
0.1023/B:GEOM.0000024780.73453.e4}}.
  \href{http://www.ams.org/mathscinet-getitem?mr=2057244}{\nolinkurl{MR2057244
  (2005e:20060)}}

\bibitem{Nie24}
Jakob Nielsen, \emph{Die isomorphismengruppe der freien gruppen}, Math. Ann.
  \textbf{91} (1924), 169--209,
  \href{http://dx.doi.org/10.1007/BF01556078}{\nolinkurl{doi:10.1007/BF0155607%
8}}.

\bibitem{Ota92}
Jean-Pierre Otal, \emph{Certaines relations d'equivalence sur l'ensemble des
  bouts d'un groupe libre}, J. London Math. Soc. \textbf{s2-46} (1992), no.~1,
  123--139,
  \href{http://dx.doi.org/10.1112/jlms/s2-46.1.123}{\nolinkurl{doi:10.1112/jlm%
s/s2-46.1.123}}.
  \href{http://www.ams.org/mathscinet-getitem?mr=1180888}{\nolinkurl{MR1180888
  (93h:20041)}}

\bibitem{Pap05}
Panos Papasoglu, \emph{Quasi-isometry invariance of group splittings}, Ann. of
  Math. (2) \textbf{161} (2005), no.~2, 759--830,
  \href{http://dx.doi.org/10.4007/annals.2005.161.759}{\nolinkurl{doi:10.4007/%
annals.2005.161.759}}.

\bibitem{Sag95}
Michah Sageev, \emph{Ends of group pairs and non-positively curved cube
  complexes}, Proc. London Math. Soc. (3) \textbf{71} (1995), no.~3, 585--617,
  \href{http://dx.doi.org/10.1112/plms/s3-71.3.585}{\nolinkurl{doi:10.1112/plm%
s/s3-71.3.585}}.
  \href{http://www.ams.org/mathscinet-getitem?mr=1347406}{\nolinkurl{MR1347406
  (97a:20062)}}

\bibitem{Sch97}
Richard~Evans Schwartz, \emph{Symmetric patterns of geodesics and automorphisms
  of surface groups}, Invent. Math. \textbf{128} (1997), no.~1, 177--199,
  \href{http://dx.doi.org/10.1007/s002220050139}{\nolinkurl{doi:10.1007/s00222%
0050139}}.

\bibitem{Sta99}
John~R. Stallings, \emph{Whitehead graphs on handlebodies}, Geometric group
  theory down under ({C}anberra, 1996), de Gruyter, Berlin, 1999, pp.~317--330.
  \href{http://www.ams.org/mathscinet-getitem?mr=1714852}{\nolinkurl{MR1714852
  (2001i:57028)}}

\bibitem{Sto97}
Richard Stong, \emph{Diskbusting elements of the free group}, Math. Res. Lett.
  \textbf{4} (1997), no.~2, 201--210.
  \href{http://www.ams.org/mathscinet-getitem?mr=1453054}{\nolinkurl{MR1453054
  (98h:20049)}}

\bibitem{Whi36a}
J.~H.~C. Whitehead, \emph{On certain sets of elements in a free group}, Proc.
  London Math. Soc. (2) \textbf{41} (1936), no.~2097, 48--56.

\bibitem{Whi36}
\bysame, \emph{On equivalent sets of elements in a free group}, Ann. of Math.
  (2) \textbf{37} (1936), no.~4, 782--800,
  \href{http://www.jstor.org/stable/1968618}{\nolinkurl{http://www.jstor.org/s%
table/1968618}}.
  \href{http://www.ams.org/mathscinet-getitem?mr=1503309}{\nolinkurl{MR1503309%
}}

\end{thebibliography}
\end{document}